\documentclass[11pt,reqno]{amsart}
\usepackage{verbatim,amsmath,amssymb, a4}
\usepackage[a4paper,left=24mm,right=24mm,top=32mm,bottom=30mm]{geometry} 

\usepackage{epsfig}
\usepackage{subfigure}
\usepackage[utf8]{inputenc}

\usepackage{xcolor}

\newcommand{\Reff}[1]{Fig. \ref{#1}}

\newcommand{\eps}{\varepsilon}
\newcommand{\D}{\nabla}
\newcommand{\pd}{\partial}
\newcommand{\R}{\mathbb{R}}            
\newcommand{\N}{\mathbb{N}}            
\newcommand{\W}{\mathcal{W}}
\newcommand{\Ad}{\mathcal{H}}
\newcommand{\M}{\mathcal{M}}
\newcommand{\F}{\mathcal{F}}
\newcommand{\Ha}{\mathcal{H}}
\newcommand{\LL}{\mathcal{L}}
\newcommand{\Chi}{\mathcal{X}}
\DeclareMathOperator{\loc}{loc}

\theoremstyle{plain}
\numberwithin{equation}{section}
\newtheorem{lemma}{Lemma}[section]
\newtheorem{theorem}[lemma]{Theorem}

\theoremstyle{definition}

\begin{document}

\title[Approximations of Willmore Flow]
{Colliding Interfaces in Old and New Diffuse-interface Approximations of Willmore-flow}
\author[S. Esedo\=glu, A. R\"atz, M. R\"oger]{Selim Esedo\=glu, Andreas R\"atz,
  Matthias R\"oger}




\date{\today}
\maketitle

\begin{abstract}
  This paper is concerned with diffuse-interface approximations of the
  Willmore flow. We first present numerical results of standard
  diffuse-interface models for colliding one dimensional
  interfaces. In such a scenario evolutions towards interfaces with
  corners can occur that do not necessarily describe the adequate sharp-interface dynamics. 

  We therefore propose and investigate alternative diffuse-interface approximations that lead to a different and more regular behavior if interfaces collide.
  These dynamics are derived from approximate energies that converge to the $L^1$-lower-semicontinuous envelope of the Willmore energy, which is in general not true for the more standard Willmore approximation.
\end{abstract}

\section{Introduction}

Diffuse-interface approximation of geometric evolution equations has a long history and is widely used in numerical simulations. One advantage of the diffuse-interface approach is that usually an automatic treatment of topological changes in a reasonable manner is guaranteed. It is however not always clear which (generalized) sharp-interface evolution is in fact approximated. We discuss here this issue in the case of the diffuse Willmore flow, in particular in  situations where collisions between different interfaces occur and where these interfaces interact with each other.
In applications such as image processing and computer vision, it is of great interest to compute Willmore flow through such topological events.

In the following we fix a nonempty open set $\Omega \subset \R^n$. Let $\M$
denote the class of open sets $E\subset\Omega $ with $\Gamma=\partial
E\cap\Omega$ given by a finite union of embedded closed
$(n-1)$-dimensional  $C^2$-manifolds without boundary in $\Omega$. We
associate to such $\Gamma$ the inner unit normal field
$\nu:\Gamma\to\R^n$,  the shape operator  $A$  with respect
to $\nu$, and  the principal curvatures $\kappa_1, \dots,
\kappa_{n-1}$ with respect to $\nu$. Finally we define the scalar mean
curvature $H=\kappa_1+\ldots + \kappa_{n-1}$ and the mean curvature vector $\vec{H}=H\nu$. 

The Willmore energy \cite{Willmore93} is then defined as
\begin{equation}
  \label{eq:wf}
  \mathcal{W}(\Gamma) := \frac1{2}\int_\Gamma H^2(x)\; \text{d}
  \mathcal{H}^{n-1}(x).
\end{equation}
The corresponding $L^2$-gradient flow is called Willmore flow. For an evolving family of sets $(E(t))_{t\in(0,T)}$ in $\M$ with boundaries $\Gamma(t)=\partial E(t)\cap \Omega$ the velocity in direction of the inner normal field $\nu(t)$ is given by
\begin{align}
	v(t) \,&=\, \Delta_{\Gamma(t)} {H}(t) - \frac{1}{2}H(t)^3 + H(t)|A(t)|^2, \label{eq:wflow}
\end{align}
on $\Gamma(t)$, where $|A(t)|^2=\kappa_1^2+\dots+\kappa_{n-1}^2$
denotes the squared Frobenius norm of the shape operator $A(t)$ and
$\Delta_{\Gamma(t)}$ denotes the Laplace-Beltrami operator on
$\Gamma(t)$.

In two dimensional space the Willmore functional for curves and the
Willmore flow are better known as Eulers elastica functional and
evolution of elastic curves. In this case \eqref{eq:wflow} reduces to 
\begin{align}
	v(t)\,&=\, \Delta \kappa(t) +\frac{1}{2}\kappa(t)^3,
\end{align}
where $\kappa(t)$ denotes the curvature of $\Gamma(t)$. The Willmore flow of a single curve in the plane exists for all times \cite{DzKS02} and converges for fixed curve length to an elastica.

A well-known and widely used diffuse-interface approximation of the Willmore energy is motivated by a conjecture of De Giorgi \cite{Gior91} and is in a modified form given by
\begin{align}
	\W_\eps(u) \,&:=\, \int_{\R^n} \frac{1}{2\eps} \Big(-\eps\Delta u + \frac{1}{\eps}W'(u)\Big)^2\,\text{d}\LL^n, \label{eq:wd}
\end{align}
where $W'$ is the derivative of a suitable double-well potential $W$ and $u$ is a smooth function on $\R^n$. The corresponding formal approximation of the Willmore flow is then given by the $L^2(\Omega)$-gradient flow of $\W_\eps$,
\begin{align}
  \eps \partial_t u \,&=\, \Delta \big(-\eps\Delta u+
  \frac{1}{\eps}W'(u)\big) - \frac{1}{\eps^2}W''(u)\big(-\eps\Delta u+
  \frac{1}{\eps}W'(u)\big), \label{eq:Wf-d}
\end{align}
complemented by suitable boundary conditions for $u$ on
$\partial\Omega$ and an initial condition for $u$ in $\Omega$.

If contact and collision of (sharp) interfaces are possible it is a priori not clear how to continue the Willmore flow. This situation already occurs for the evolution of several curves in the plane. In many applications interactions of different curves should be considered and treating the evolution of each curve separately might not be appropriate. To account for touching and colliding curves in such situations a suitable generalization of the evolution beyond the smooth embedded case is required. One possible extension is a gradient dynamic with respect to a suitably relaxed Willmore functional for general sets. A natural candidate for such relaxation is the $L^1(\Omega)$-lower-semicontinuous envelope
\begin{gather}
	\tilde\W(E) \,:=\, \inf \Big\{\liminf_{k\to\infty} \W(\partial E_k) \,:\, E_k\,{\overset{k\to\infty}{\longrightarrow}}\, E\text{ in }L^1(\Omega), E_k \in \M\text{ for all }k\in\N\Big\} \label{eq:genE}.
\end{gather}
It is however difficult to characterize and numerically approximate the corresponding gradient flow. Instead we consider here diffuse-interface approximations that naturally allow for collision of (diffuse) interfaces and exist globally in time. 

Our first observation is that the usual diffuse approximation \eqref{eq:wd} does not approximate any gradient flow with respect to the generalized elastic energy \eqref{eq:genE}. In fact, careful numerical simulations show that the diffuse evolution in the plane can lead to transversal intersections of interfaces, which is known to have infinite energy with respect to the $L^1(\R^2)$-lower-semicontinuous envelope of the elastica functional \cite{BeMu04}. The occurrence of intersections is related to the existence of saddle solutions of the Allen-Cahn equation, as we will explain below. This example also demonstrates that already on the level of energies a discrepancy between the diffuse elastica energy and the relaxation \eqref{eq:genE} of the sharp-interface energy occurs. In particular, $\W_\eps$ does not Gamma-converge to $\tilde\W$ \cite{Mu10}.

We will next discuss two diffuse-interface approximations of the Willmore flow that Gamma-converge to $\tilde\W$. The first one was proposed by Bellettini \cite{Bell97}. As an alternative we introduce a modification of the usual diffuse Willmore functional that adds a penalty term to $\W_\eps$. This penalty term is very small as long as no collisions of interfaces occur, allows for the touching of diffuse interfaces, but prevents transversal intersections. For both approaches we consider the $L^2$-gradient flows and present numerical simulations. 

\section{Diffuse approximation of the Willmore functional and diffuse Willmore flows}
\label{sec:intro-willmore}
Diffuse-interface approximations of the Willmore energy
\eqref{eq:wf} are usually based on a Ginzburg--Landau free energy for a
phase-field variable $u$
\begin{equation}
  \label{eq:gl}
  \Ad_\eps(u) := \int_{\Omega}\left( \frac{\eps}{2} |\D u|^2 +
    \eps^{-1} W(u)\right)  \;\text{d}x,
\end{equation}
where $W$ denotes a suitable double-well potential that we choose in the following as $W(u) = 18 u^2(1-u)^2$. For a smooth phase field $u:\Omega\to\R$ let us further define the $L^2$-gradient of $\Ad_\eps$,
\begin{equation*}
  w := \frac{\delta \mathcal{H}_\eps}{\delta u} = -\eps \Delta u +
  \eps^{-1}W'(u),
\end{equation*}
the diffuse normal vector
\begin{align*}
	\nu(x)\,:=\, \begin{cases}
		\frac{\nabla u}{|\nabla u|} &\text{ if } \nabla u\neq 0,\\
		0 &\text{ else,}
		\end{cases}
\end{align*}
and the \emph{level set mean curvature}
\begin{align}
	v(x) \,:=\, \nabla \cdot \nu(x). \label{eq:levelH}
\end{align}
For phase fields $u$ with `moderate' energy $\Ad_\eps(u)$ the function $u$ will look like a smoothened indicator function that is close to the values $0,1$ in a large part of the domain and possibly with thin transition layers. The width of these \emph{diffuse interfaces} is proportional to $\eps>0$. The
energy $\Ad_\eps$ is a diffuse-interface counterpart of the
surface area functional. This statement was made precise by Modica and Mortola, who proved the Gamma-convergence with respect to $L^1$ of $\Ad_\eps$ to the perimeter functional \cite{MoMo77,Modi87}. The $L^2$-gradient $w$ of $\Ad_\eps$,
describes a kind of \emph{diffuse mean curvature} and motivates the definition of the diffuse Willmore functional \eqref{eq:wd}. The approximation $\W_\eps$ has been studied intensively and is widely used in numerical simulations \cite{LoMa00,BeMu05,DuLiWa04,LoRaVo09}. For space dimension $n=2,3$ it is known \cite{RoeS06} that (for uniformly bounded $\Ad_\eps$) the functionals $\W_\eps$ Gamma-converge towards the Willmore functional in limit points $E \subset \Omega$ with $C^2$-boundary in $\Omega$. The Gamma-convergence is not true in general limit points $E\subset\R^n$. In fact \cite{DaFP92} showed the existence of a smooth function $u:\R^2\to (-1,1)$ (where $\pm 1$ are the zeros of the double well potential) with the following properties: $u$ is a saddle solution of the Allen-Cahn equation
\begin{gather*}	
	- \Delta u +W'(u) \,=\, 0\quad\text{ in }\R^2,
\end{gather*}
$u=0$ holds on the coordinate axes $\{(x_1,x_2)\,:\, x_1x_2=0\}$, and $u$ is positive inside the first and third quadrant of $\R^2$ and negative inside the second and fourth quadrant. Moreover, there exists $k>0$ such that for $|x|,|y|>k$ the saddle solution $u$ is exponentially close to $\pm 1$ and $|\nabla u|$ is exponentially small. In particular, the rescaled functions $u_\eps:\R^2\to\R$, $u_\eps(x):=u(\frac{1}{\eps}x)$ have on compact subsets of $\R^2$ uniformly bounded energy $\Ad_\eps$, converge in $L^1_{\loc}(\R^2)$ to the $\pm1$-characteristic function of the set $E:=\{(x_1,x_2): x_1x_2>0\}$, and finally satisfy $\W_\eps(u_\eps)=0$ for all $\eps>0$. On the other hand, we have by \cite{BeMu04} that $\tilde{\W}(E)=\infty$ for the  $L^1_{\loc}(\R^2)$-lower-semicontinuous envelope of the Willmore functional. Therefore 
\begin{gather*}
	\infty\,=\, \tilde{\W}(E) \,>\, \liminf_{\eps\to 0} W_\eps(u_\eps)\,=\,0
\end{gather*}
which shows that $\W_\eps$ cannot Gamma-converge to $\tilde{\W}$.

For a tighter convergence of approximations some control on the Willmore energy of level sets of the phase fields is necessary. Bellettini \cite{Bell97} proposed such type of approximations for general geometric functionals. The corresponding approximation of the Willmore functional is given by the square integral of the mean curvature of the level sets of the phase field $u$ integrated with respect to the diffuse area density,
\begin{equation}
  \hat{\mathcal{W}}_\eps(u) := \frac1{2}\int_{\Omega\setminus \{|\nabla u|=0\}} \Big(\nabla \cdot
  \frac{\nabla u}{|\nabla u|}\Big)^2 \Big(\frac{\eps}{2} |\nabla u|^2 +
  \eps^{-1}W(u)\Big)\;\text{d}x. \label{eq:bell}
\end{equation}
In the particular case of \eqref{eq:bell} Bellettinis results imply that (for uniformly bounded diffuse surface area $\Ad_\eps$) the functionals $\hat\W_\eps$ in fact Gamma-converge with respect to $L^1(\Omega)$ to $\tilde\W$.

For $n = 2$ an alternative approximation of the elastica functional has been investigated by Mugnai \cite{Mu10}. He uses an approximation of the square integral of the second fundamental form,
\begin{equation}
  \label{eq:wplusadM}
  \bar{\mathcal{W}}_\eps(u) := \frac1{2 \eps}\int_\Omega\left|\eps D^2
  u - \eps^{-1}W'(u)\frac{\nabla u}{|\nabla u|} \otimes \frac{\nabla
    u}{|\nabla u|} \right|^2 \;\text{d}x.
\end{equation}
and obtains for $n=2$ the $L^1(\Omega)$-Gamma-convergence  of
$\bar{\mathcal{W}}_\eps $ to $\tilde{\mathcal{W}}$, again under a uniform bound on the diffuse surface area.


\section{A new diffuse-interface approximation of the Willmore
  functional}
\label{sec:diff_new}
We propose here a modification of the `standard' approximation $\W_\eps$ of the Willmore functional by an additional penalty term. This has some advantages in numerical simulations as we discuss below. The example of saddle solutions for the Allen-Cahn equation and a comparison with $\hat\W_\eps, \bar\W_\eps$ reveals that the standard approximation works well as long as phase fields $u$ behave like the optimal profile $q$ for the one-dimensional transition from $0$ to $1$ given by
\begin{align}
	-q'' + W(q) \,&=\, 0, \quad q(0)\,=\,\frac1{2},\quad \lim_{r\to -\infty} q(r) \,=\,0,\quad \lim_{r\to -\infty} q(r) \,=\,1. \label{eq:opt-prof}
\end{align}
Formal asymptotic expansions often use that $u_\eps(x)\approx
q(\frac{d}{\eps})$, where $d$ denotes the signed distance functions to
a limit hypersurface. This property can be formalized as  vanishing of
the \emph{discrepancy} 
\[
\zeta_\eps:= \frac{\eps}{2}|\nabla u|^2-\eps^{-1}W(u)
\]
 that measures deviation from equi-distribution in the diffuse surface energy. In the limit $\eps\to 0$ this quantity vanishes in $L^1(\Omega)$ for sequences of phase fields that have uniformly bounded diffuse surface energy and diffuse Willmore energy $\W_\eps$ \cite{RoeS06}.  This weak control does however not exclude formation of transversal intersections in the limit. Rewriting the diffuse mean curvature in terms of the level set mean curvature $v:= \nabla \cdot \frac{\nabla u}{|\nabla u|}$ as
\begin{align*}
	w\,=\, -\eps\Delta u +\frac{1}{\eps}W'(u) \,=\, -\eps |\nabla u|v -\nabla\zeta_\eps\cdot \frac{\nabla u}{|\nabla u|^2}\qquad\text{ on }\{\nabla u\neq0\}
\end{align*}
we see that the diffuse mean curvature controls the level set mean curvature if and only if the normal projection of the gradient of the discrepancy is small. 

This motivates to introduce a penalty term of the form
\begin{equation}
  \label{eq:energyAdd}
  \mathcal{A}_\eps(u) :=
  \frac1{2\eps^{1+\alpha}}\int_\Omega \left(w + \left(\eps|\nabla u|\sqrt{2W(u)}\right)^{\frac{1}{2}}v \right)^2 \;\text{d}x,
\end{equation}
where $0\leq \alpha \le 1$. If $u\approx q(\frac{d}{\eps})$ we find that  $\mathcal{A}_\eps(u) = \mathcal{O}(\eps)$ remains small. The energy however becomes large if $u$ deviates from the optimal profile structure. We then define the \emph{modified diffuse Willmore energy }
\begin{equation}
  \mathcal{F}_\eps(u) \,:=\, \mathcal{W}_\eps(u) + \mathcal{A}_\eps(u). \label{eq:def-mod-W}
\end{equation}
Using Bellettinis result \cite{Bell97} we can prove the Gamma-convergence of the modified Willmore energy $\F_\eps$ to $\tilde{\W}$. In the following we extend the functionals $\Ad_\eps$ and $\F_\eps$ to $L^1(\Omega)$ by setting them to $+\infty$ on $L^1(\Omega)\setminus C^2(\Omega)$. 
\begin{theorem}\label{thm:conv}
Let $\alpha>0$. Then the functional $\F_\eps $ Gamma-converges with respect to $L^1(\Omega)$ to $\tilde{W}$ in the following sense:
\begin{enumerate}
\item\label{it:one}	Let $(u_\eps)_{\eps>0}$ be a sequence of smooth phase fields $u_\eps:\Omega\to \R$ with
\begin{align}
	\sup_{\eps>0} \Ad_\eps(u_\eps) \,<\, \infty \label{eq:bound}
\end{align}
and $u_\eps\,\to\, u$ in $L^1(\Omega)$. Then $u\in BV(\Omega;\{0,1\})$, and
\begin{align}
	\tilde\W(E) \,\leq\, \liminf_{\eps\to 0} \F_\eps(u_\eps), \label{eq:liminf}
\end{align}
where $E=\{u=1\}$.\\
\item	Let $E\subset \Omega$ be given with $\tilde\W(E)<\infty$. Then there exists a sequence $(u_\eps)_{\eps>0}$ of smooth phase fields in $\Omega$ such that $u_\eps\to \Chi_E$ in $L^1(\Omega)$ and
\begin{align}
	\tilde\W(E) \,\geq\, \limsup_{\eps\to 0} \F_\eps(u_\eps). \label{eq:limsup}
\end{align}
\end{enumerate}
In the situation of \eqref{it:one} for $\alpha=0$ we still obtain
\begin{align}
	\frac{1}{2}\tilde\W(E) \,\leq\, \liminf_{\eps\to 0} \F_\eps(u_\eps) \label{eq:liminf2}
\end{align}
and in particular $\tilde{\W}(E)<\infty$ if the right-hand side is finite.
\end{theorem}
\begin{proof}
\begin{enumerate}
\item
By Young's inequality we deduce for any $K>0$ the following estimate.
\begin{align}
	&w^2 + K \left(w + \left(\eps|\nabla u_{\eps}|\sqrt{2W(u_{\eps})}\right)^{\frac{1}{2}}v \right)^2 \notag\\
	=\,& (1+K) w^2 +2Kw   \left(\eps|\nabla u_{\eps}|\sqrt{2W(u_{\eps})}\right)^{\frac{1}{2}}v + K\eps |\nabla u_{\eps}|\sqrt{2W(u_{\eps})}v^2 \notag\\
	\geq\,& \eps\frac{K}{1+K} |\nabla u_{\eps}|\sqrt{2W(u_{\eps})}v^2. \label{eq:young}
\end{align}
This implies
\begin{align}
	\F_\eps(u_{\eps}) \,&\geq\, \frac{\eps^{-\alpha}}{2+2\eps^{-\alpha}} \int_\Omega |\nabla u_{\eps}|\sqrt{2W(u_{\eps})}v^2\,\text{d}x \notag\\
	&\geq \frac{\eps^{-\alpha}}{2+2\eps^{-\alpha}} \int_0^1 \sqrt{2W(s)} \int_{\{u_{\eps}=s\}\cap \{\nabla u_{\eps}\neq 0\}} v^2(x)\,\text{d}\Ha^{n-1}(x)\,\text{d}s, \label{eq:est-bell}
\end{align}
where we have used the co-area formula in the last line.\\
For $(u_\eps)_{\eps>0}$ with $u_\eps\to u$ in $L^1(\Omega)$ we first have by the Modica--Mortola Theorem that $u\in BV(\Omega;\{0,1\})$ with
\begin{align*}
	\int_\Omega |\nabla u| \,\leq\, \liminf_{\eps\to 0} \Ad_\eps(u_\eps).
\end{align*}
After passing to a subsequence $(\eps_k)_{k\in\N}$ we may assume 
\begin{align*}
	\lim_{k\to\infty} \F_{\eps_k}(u_{\eps_k})\,&=\, \liminf_{\eps\to 0} \F_\eps(u_\eps).
\end{align*}
By \eqref{eq:est-bell} we now can follow the proof of \cite[Thm. 4.2]{Bell97}. First one obtains a subsequence $k\to\infty$ (not relabeled) and a set $I\subset (0,1)$ with full measure such that for any $s\in I$
\begin{gather*}
	\{u_{\eps_k}=s\}\,=\,\partial\{u_{\eps_k}>s\},\\
	\{u_{\eps_k}=s\}\cap \{\nabla u_{\eps_k}= 0\}\,=\,\emptyset,\\
	 \Chi_{\{u_{\eps_k}>s\}}\,\to\, \Chi_{\{u>s\}}=\Chi_E\text{ as }k\to\infty,
\end{gather*}
where $E=\{u=1\}$.
Moreover by the definition of $\tilde\W$ we have
\begin{align*}
	\tilde\W(E) \,&\leq\, \liminf_{k\to\infty} \W(\Chi_{\{u_{\eps_k}>s\}})\,=\, \liminf_{k\to\infty} \int_{\{u_{\eps_k}=s\}} v^2\,\text{d}\Ha^{n-1}
\end{align*}
for any $s\in I$. By \eqref{eq:est-bell} and Fatou's Lemma for $\alpha>0$ we eventually obtain
\begin{align*}
	\liminf_{\eps\to 0} \F_\eps(u_\eps) \,&\geq
        \tilde{\W}(u)\int_0^1 \sqrt{2W(s)} \,\text{d}s \,=\,  \tilde{\W}(u).
\end{align*}
In the case $\alpha=0$ the same argument shows \eqref{eq:liminf2}.
\item
Let first $E\subset\Omega$ have smooth boundary. We follow the standard construction of a recovery sequence and  consider the one-dimensional optimal profile $q$ from \eqref{eq:opt-prof} and the signed distance function $d$ from $\partial E$. We then set $u_\eps := q(\frac{d}{\eps})$ in $\{|d|<\delta\}$ where $\delta>0$ is suitably small such that the projection on $\partial E$ is smooth on $\{|d|<2\delta\}$. In $\{|d|>2\delta\}$ we set $u_\eps$ to $1$ in $E$ and $0$ in $\Omega\setminus E$. In $\{\delta<|d|<2\delta\}$ we choose $u_\eps$ to smoothly interpolate in such a way that $\W(u_\eps)$ and $|\nabla u_\eps|$ are exponentially small in $\{|d|\geq \delta\}$, see for example \cite[Section 4]{DoMR11}. Then the Willmore energy $\W_\eps(u_\eps)$ is known to converge to $\W(E)$. For the additional part in the energy we compute in the set $\{|d|<\delta\}$
\begin{align*}
	w  + \left(\eps|\nabla u|\sqrt{2W(u)}\right)^{\frac{1}{2}}v \,&=\,  -q'(\frac{d}{\eps})\Delta d + (\sqrt{2W(q(\frac{d}{\eps}))}q'(\frac{d}{\eps}))^\frac{1}{2}v\,
	=\, 0,
\end{align*}
since in $\{|d|<\delta\}$ we have $\sqrt{2W(q)}=q'$ and $v=\Delta d$ as level sets of $u_\eps$ correspond in that region to level sets of $d$.

In the region $\{|d|\geq 2\delta\}$ we have $w,\eps|\nabla u|\sqrt{2W(u)}=0$.  Finally, for a carefully chosen interpolate in the construction of $u_\eps$, in $\{\delta\leq |d|\leq 2\delta\}$ we have that $w$ and $\eps|\nabla u|\sqrt{2W(u)}$ are exponentially small. Furthermore, $\Delta d =v$ is controlled in terms of the principal curvatures of $\partial E$. This shows the approximation property for $C^2$-boundaries $\partial E$.

To deal with the general case we only need to show that $\tilde{\W}=\W$ on sets with $C^2$-boundaries. This condition is in fact satisfied by \cite{Sch09} and \cite{Menn08a}.     
\end{enumerate}
\hfill
\end{proof}

\section{Numerical simulations for the standard diffuse Willmore flow}
Numerical investigations of the Willmore flow and related dynamics mainly build on parametric (sharp-interface) approaches and implicit
treatments by level set and phase-field methods. Parametric approaches
for the Willmore flow have been proposed in \cite{BaGaNu07,BaGaNu10} for
curves and in \cite{Ru05,DeDzEl05,BaGaNu08} for curves and
surfaces. Generalized Helfrich--type flows for single- and
multicomponent vesicles have been studied with sharp-interface methods
in \cite{BoNoPa10} and \cite{ElSt10}, respectively. Level set methods
have been applied first in \cite{DrRu04} to the Willmore flow. For a
comparison of level set and sharp-interface approaches we refer to
\cite{BeMiObSe09}. Phase field approximations for the Willmore flow have
been numerically investigated in \cite{DuLiWa04}. For diffuse-interface
approximations of Helfrich-type flows, we refer to
\cite{BiKaMi05,DuLiWa06,CaHeMa06}. Coupled
Helfrich- and Cahn-Hilliard-type flows have been numerically
treated with phase-field models in \cite{WaDu08,LoRaVo09}.

In this section, we consider the standard diffuse Willmore flow \eqref{eq:wd} and focus on situations where (diffuse) interfaces collide. We present numerical results for both a finite element
discretization and a finite difference scheme of the diffuse-interface flow
\eqref{eq:Wf-d}. 

\subsection{Finite element approximation}
The discretization is implemented in
the FEM toolbox AMDiS \cite{amdis}. First, we use linear elements in
space and a semi-implicit time discretization with a linearization of
nonlinearities. Furthermore, we use a uniform grid and discretize
\eqref{eq:Wf-d} as a coupled system of two second order
PDE's for the discrete solutions $u_h$ and $w_h$ and solve the
resulting linear system with a direct solver (UMFPACK,
\cite{umfpack}). Thereby, we consider a
domain $\Omega = (-1,1)^2 \subset \R^2$ and assume periodic solutions
at the boundary $\partial \Omega$. Moreover, we use a simple adaptive
strategy in time, where time steps $\Delta t_{m} \in [10^{-7},5 \cdot
10^{-6}]$ are inversely proportional to the maximum of the discrete
time derivative of the phase-field variable. For the results presented
in this section we have used $\varepsilon = 0.1$.

\subsubsection{Symmetric initial condition}
\label{sec:symmetricinitialcondition}
As initial condition, we have chosen a phase-field function
$u_h(\cdot,0):\overline \Omega \to \R$ having nine symmetrically
distributed circular levelsets $\{u_h(\cdot, 0) = 1/2\}$ (\Reff{fig1},
left) with equal radii $0.1$. In \Reff{fig1}, one can see the contour
plots of $u_h$ at different times. Thereby discs start to grow until
the interfaces begin to ``feel'' each other. Then the interface forms
sharp corners. Our interpretation of such behavior is that the diffuse approximations converge to the saddle solution of the Allen--Cahn equation, discussed in Section \ref{sec:intro-willmore}. As the diffuse mean curvature for saddle solutions vanishes, such corners carry very little diffuse Willmore energy.
\begin{figure}[here]
\begin{center}
\includegraphics*[width=0.24\textwidth]{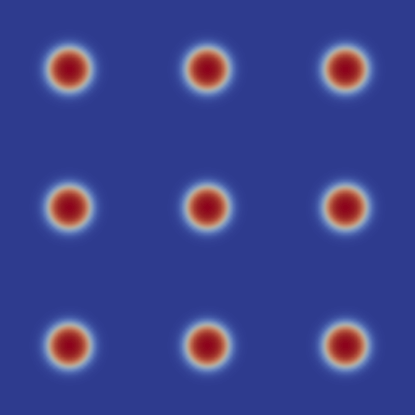}
\includegraphics*[width=0.24\textwidth]{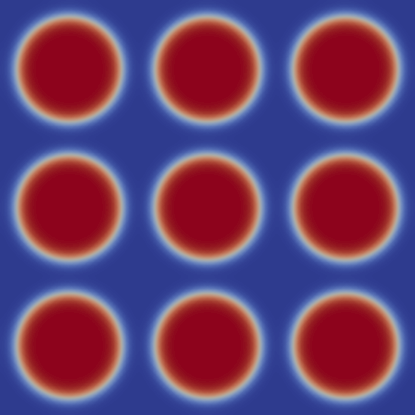}
\includegraphics*[width=0.24\textwidth]{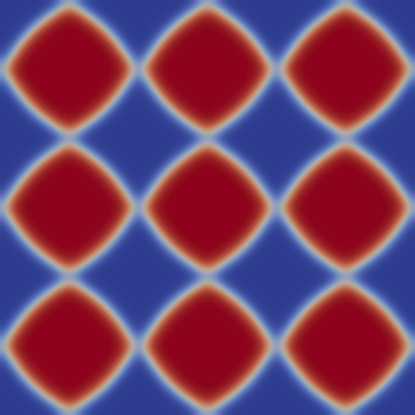}
\includegraphics*[width=0.24\textwidth]{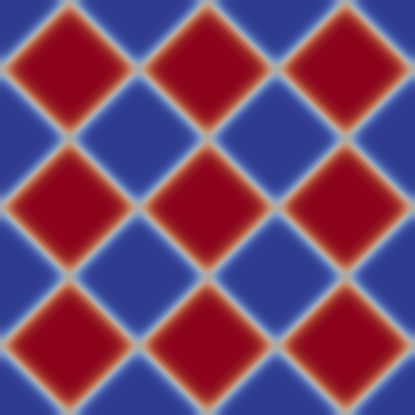}
\caption{\footnotesize \label{fig1} Evolution of ``standard'' diffuse-interface
  Willmore flow \eqref{eq:Wf-d}: Discrete phase-field $u_h$ for different times
  $t=0$, $t\approx0.0019$, $t\approx0.0024$, $t\approx0.0037$.}
\end{center}
\end{figure}

\subsubsection{Non-symmetric initial condition}

In \Reff{fig2}, similar phenomena can be observed for non symmetric
initial conditions. The discrete Willmore energy is plotted in
\Reff{fig3}. For large times $t$ this energy 
\begin{equation*}
  e(t) := \mathcal{W}_\eps(u_h(\cdot , t))
\end{equation*}
is close to $0$.

\begin{figure}[here]
\begin{center}
\includegraphics*[width=0.24\textwidth]{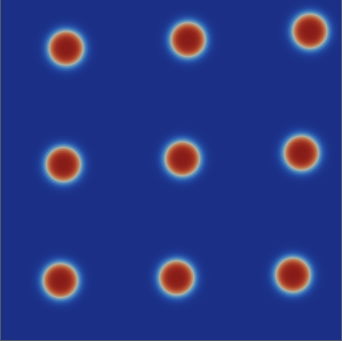}
\includegraphics*[width=0.24\textwidth]{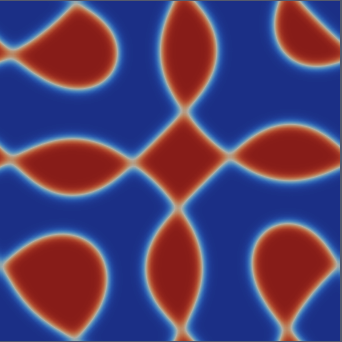}
\includegraphics*[width=0.24\textwidth]{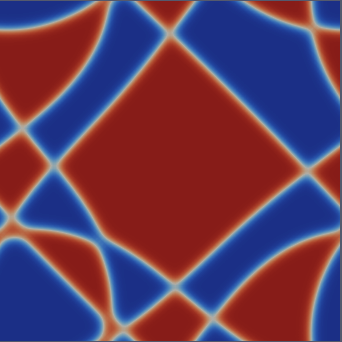}
\includegraphics*[width=0.24\textwidth]{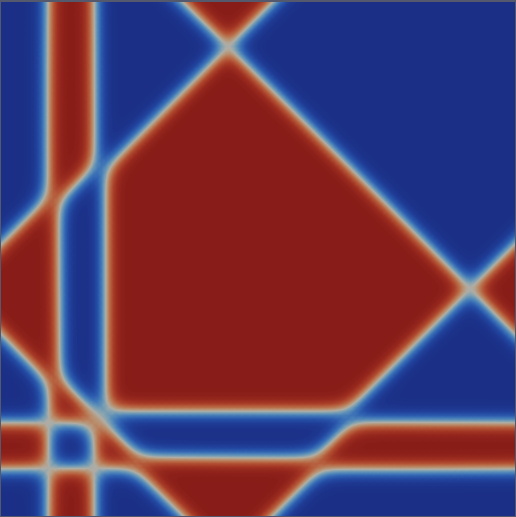}
\caption{\footnotesize \label{fig2} Evolution of ``standard'' diffuse-interface
  Willmore flow \eqref{eq:Wf-d}: Discrete phase-field $u_h$ for different times
  $t=0$, $t\approx0.0039$, \mbox{$t\approx0.1038$,} $t\approx0.3338$.}
\end{center}
\end{figure}

\begin{figure}[here]
\begin{center}
\includegraphics*[width=0.5\textwidth]{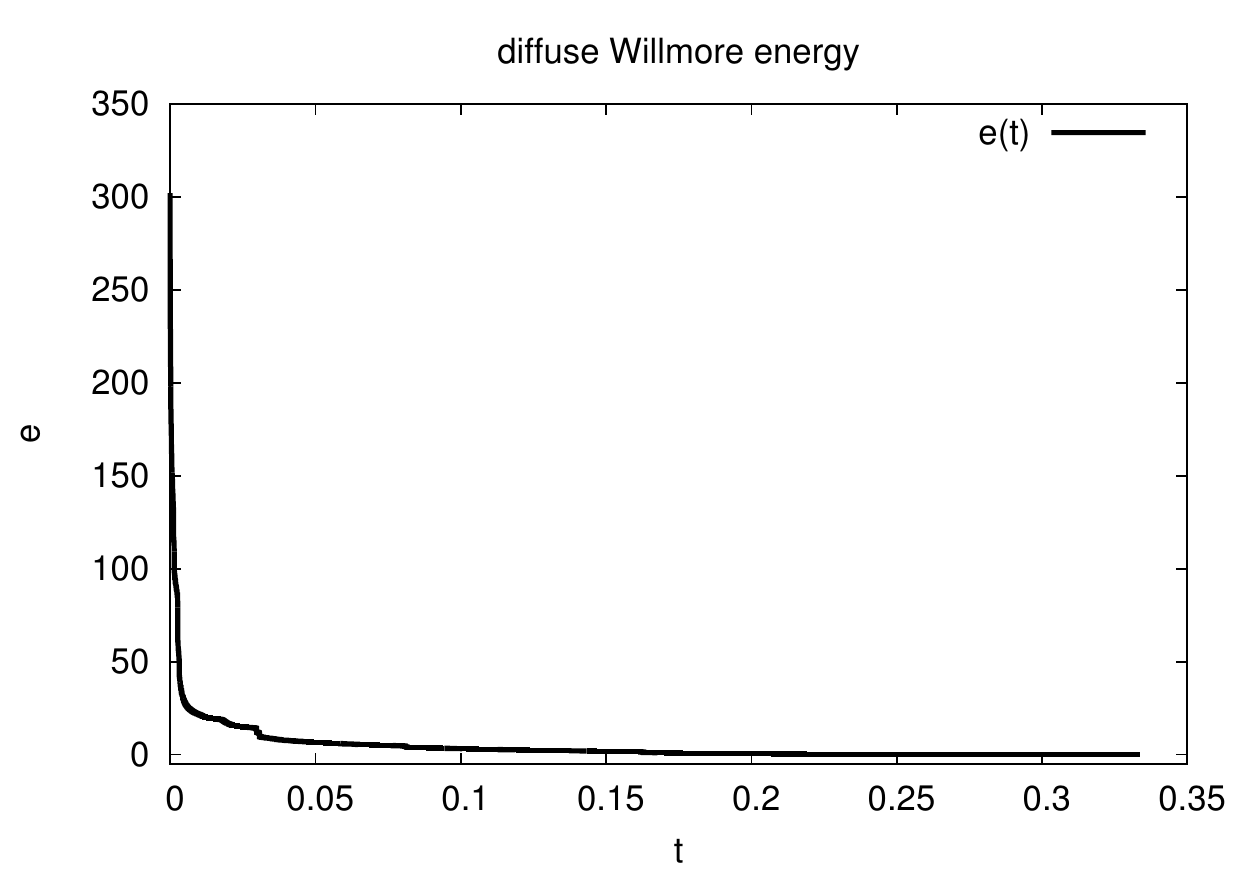}
\caption{\footnotesize \label{fig3} Evolution of ``standard'' diffuse-interface
  Willmore flow \eqref{eq:Wf-d}: Diffuse Willmore energy $e(t) =
  \mathcal{W}_\eps(u_h(\cdot, t))$ versus time $t$.}
\end{center}
\end{figure}

\subsubsection{Two circles initial condition}
\label{sec:twocirclesinitialcondition}
As a further example, we consider an initial condition with two circles
with radii $0.2$ and $0.3$. Contour plots of $u_h$ at different times
are shown in Fig. \ref{fig:wil2}, where after collision of interfaces
a transversal intersection appears.

\begin{figure}[here]
\begin{center}
\includegraphics*[width=0.24\textwidth]{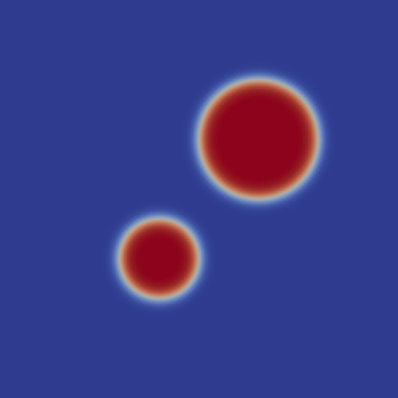}
\includegraphics*[width=0.24\textwidth]{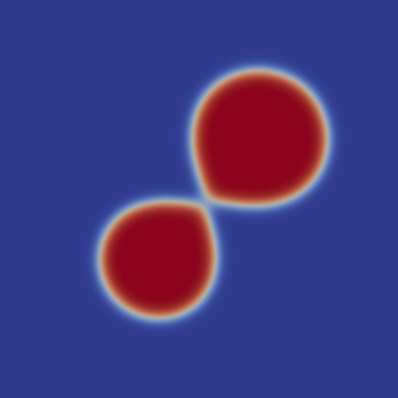}
\includegraphics*[width=0.24\textwidth]{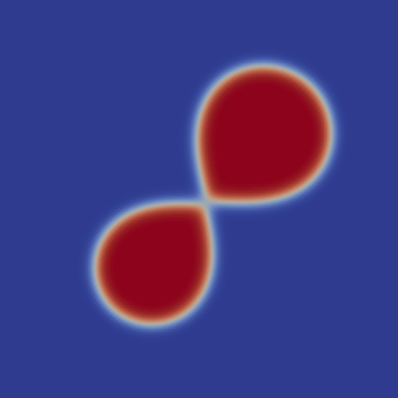}
\includegraphics*[width=0.24\textwidth]{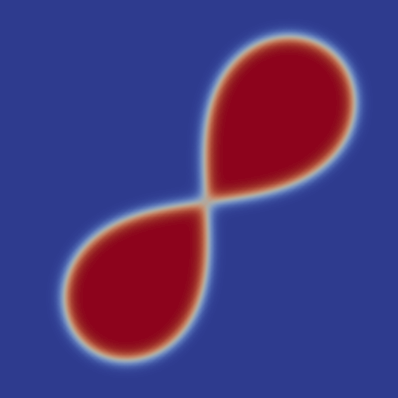}
\caption{\footnotesize \label{fig:wil2} Evolution of ``standard'' diffuse-interface
  Willmore flow \eqref{eq:Wf-d}: Discrete phase-field $u_h$ for different times
  $t=0$, $t\approx0.003$, \mbox{$t\approx0.0045$,} $t\approx0.0200$.}
\end{center}
\end{figure}

\subsection{Finite differences implementation}
\label{sec:finitedifferencesimplementation}
In this subsection, we implement the standard diffuse Willmore flow using finite differences, and observe the same unexpected behavior in 2D as in the previous subsections: the colliding interfaces form a ``cross'', even though this should be precluded by the sharp-interface energy.
Therefore, it seems unlikely that this surprising behavior is due to a numerical artifact: it appears to be an intrinsic feature of the standard approximation (\ref{eq:wd}).
However, we also provide simulations in three dimensions that suggest that this phenomenon might be two dimensional only:
The standard diffuse-interface model inspired by De Giorgi's conjecture appears to lead to a topological change (mergers) when two typical surfaces collide in 3D.
This is in keeping with some of the numerical experiments carried out by Du et. al. in \cite{DuLiWa04,DuLiWa06}.
Whether Gamma-convergent numerical approximations, such as the one due to Bellettini \cite{Be97} or the new one presented in this paper in Section \ref{sec:diff_new}, also lead to a topological change in 3D under these circumstances will be investigated subsequently in Section \ref{sec:bellettini_numerics}. 
See also Section \ref{sec:discussion} for a discussion of why topological changes are in fact more likely in 3D than in 2D for gradient descent of the $L^1$ relaxation of Willmore energy.

In 2D, the results presented in Figure \ref{fig:standard} are for the diffuse-interface energy
$$ \mathcal{W}_\varepsilon (u) + \gamma \mathcal{H}_\varepsilon (u), $$
where the second term is included to ensure a uniform bound of the diffuse surface area (note however that for evolutions with well-behaved initial conditions this is often automatically satisfied, as for example in the simulations above).\\
Our scheme for its $L^2$ gradient flow is:
\begin{equation}
\label{eq:findiffdegiorgi}
\frac{u^{n+1}-u^n}{\delta t} = -\Delta_h w^n + \left\{ \frac{1}{\varepsilon^2} W''(u^n) + \gamma \right\} w^n
\end{equation}
where
\begin{equation}
\label{eq:v}
w^n := \varepsilon\Delta_h u^n- \frac{1}{\varepsilon} W'(u^n)
\end{equation}
and $\Delta_h$ is the standard centered differences discretization of the Laplacian on a uniform grid.
This is an explicit time stepping scheme, the stability (CFL) condition (upper bound on the time step size $\delta t$) for which scales as $(\delta x)^4$ as $\delta x\to 0^+$.

Alternatively, we have the following semi-implicit version:
\begin{equation}
\label{eq:semiimplicit}
\frac{u^{n+1}-u^n}{\delta t} = -\varepsilon \Delta_h^2 u^{n+1} + \frac{1}{\varepsilon} \Delta_h \Big( W'(u^n) \Big) + \left\{ \frac{1}{\varepsilon^2} W''(u^n) + \gamma \right\} w^n
\end{equation}
where $w^n$ is again as in (\ref{eq:v}).
At each time step, $u^{n+1}$ is solved for via the discrete Fourier transform.
This scheme appears to be stable for much larger time step sizes than (\ref{eq:findiffdegiorgi}).
In the interest of minimizing the possibility of numerical artifacts, we refrain from further attempts to improve the computational efficiency of the schemes used here, even though there is no shortage of classical techniques for doing so.

Numerical simulations using scheme (\ref{eq:semiimplicit}) are shown in Figure \ref{fig:standard}.
The computational domain was $[0,1]^2$, and the diffuse-interface parameter was chosen to be $\varepsilon = 0.03$.
The spatial resolution was $200\times 200$.
The parameter $\gamma$ was taken to be $\frac{1}{4}$.
The results testify to the same surprising qualitative behavior as in simulations of Sections \ref{sec:symmetricinitialcondition} and \ref{sec:twocirclesinitialcondition}.
Thus, this appearance of crossings when interfaces collide appears to be a robust, inherent feature of the standard diffuse-interface approximation.

\begin{figure}[h]
\begin{center}
\includegraphics[width=0.24\textwidth]{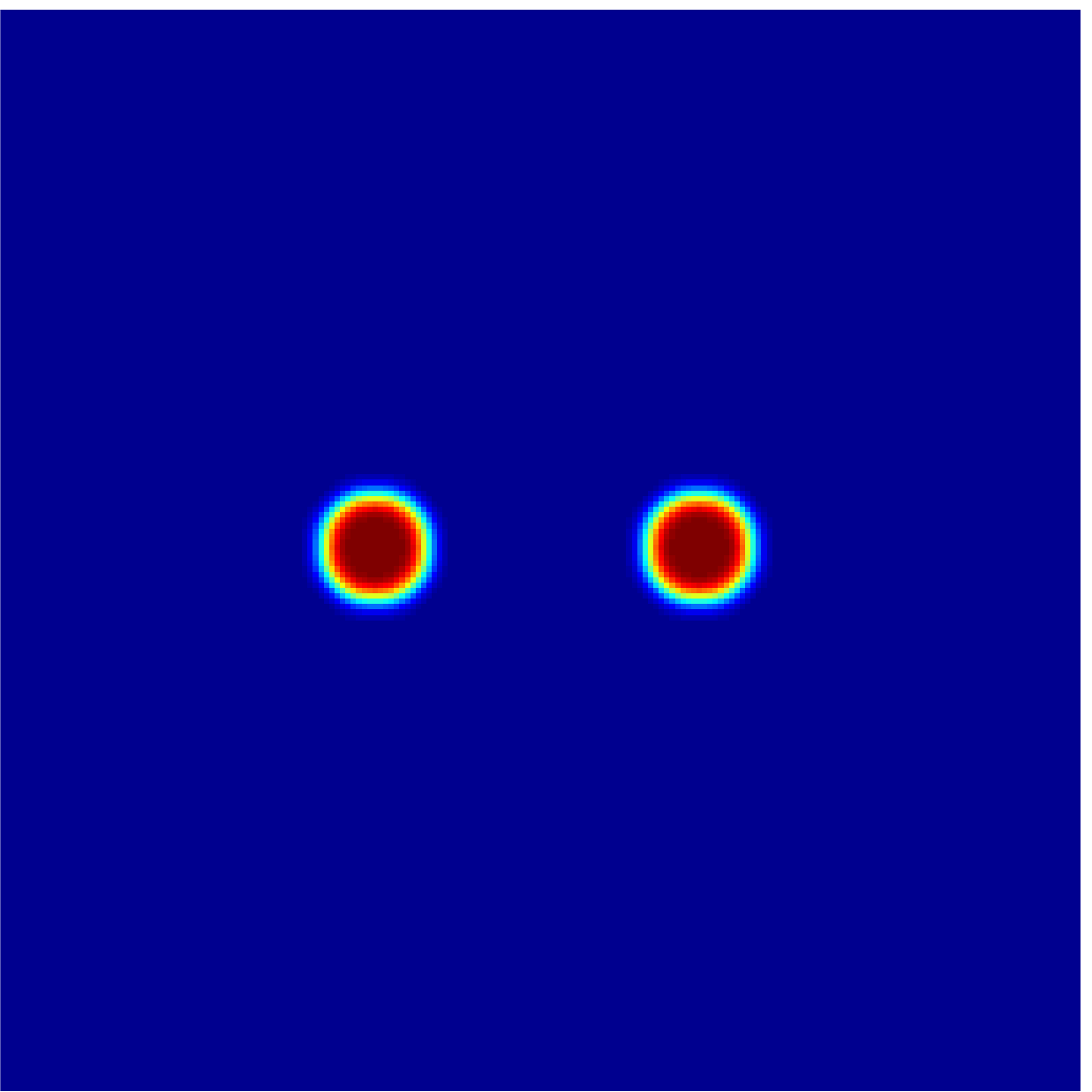}
\includegraphics[width=0.24\textwidth]{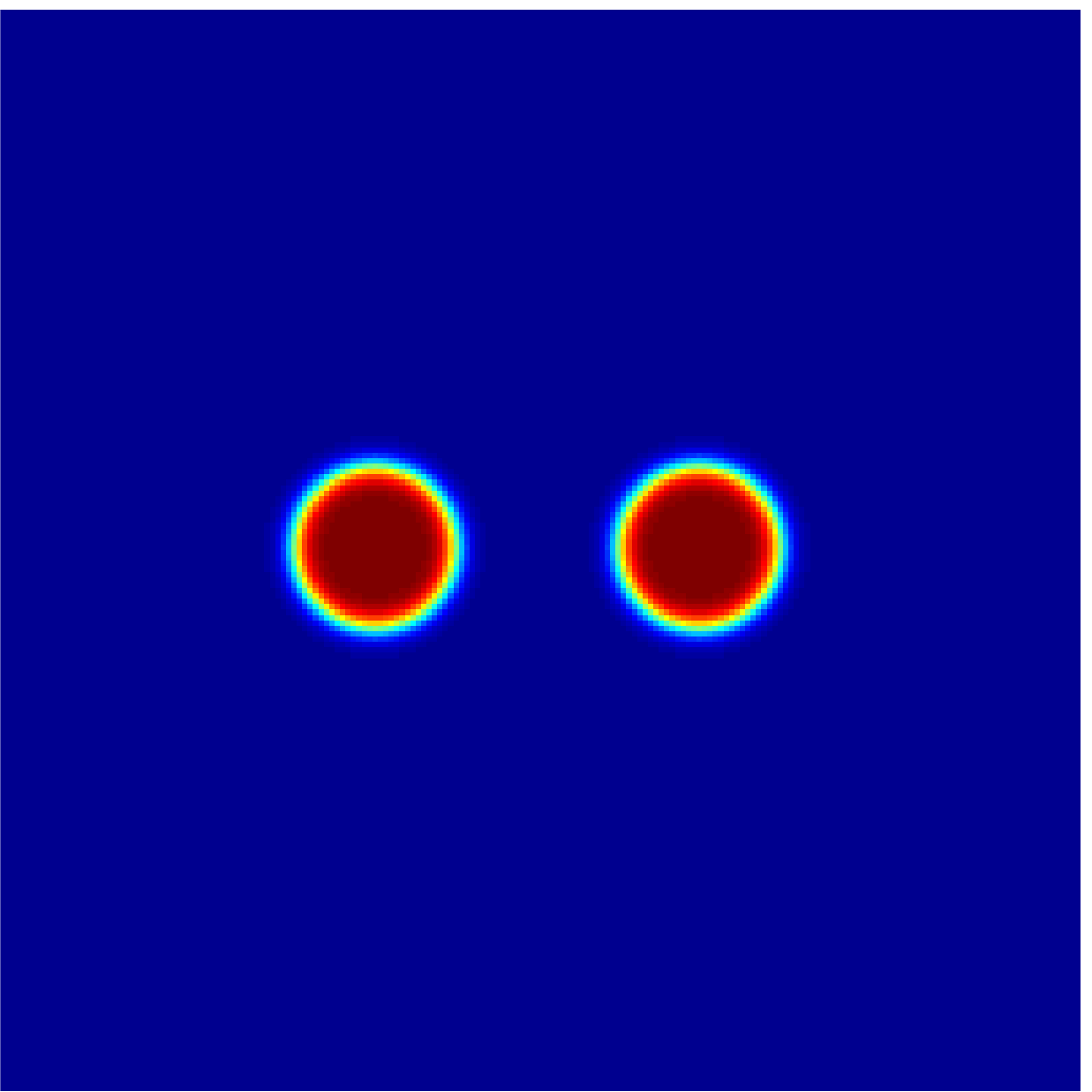}
\includegraphics[width=0.24\textwidth]{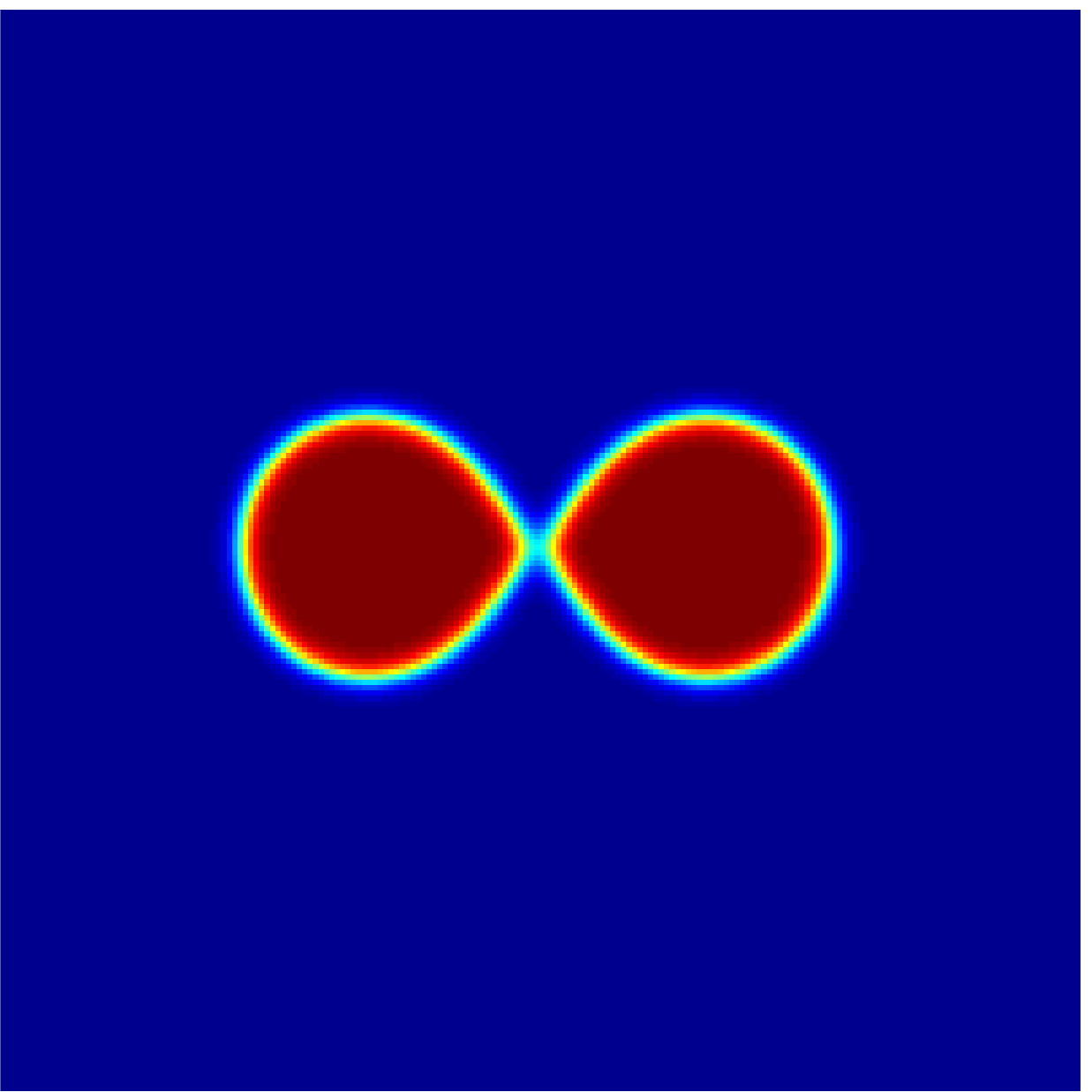}
\includegraphics[width=0.24\textwidth]{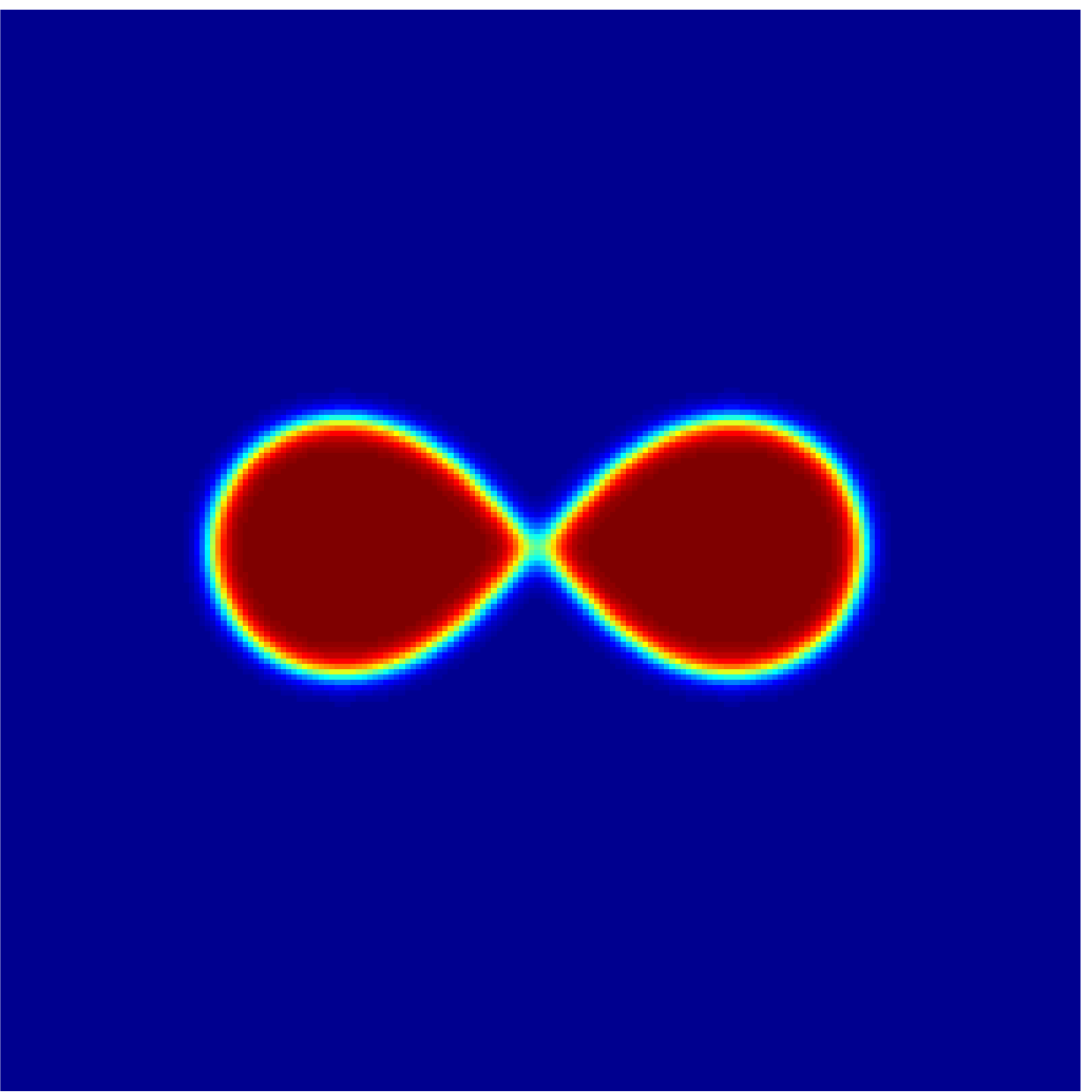}
\caption{\footnotesize Gradient flow for the standard diffuse-interface approximation of Willmore energy, using the finite differences scheme \ref{eq:semiimplicit}.}
\label{fig:standard}
\end{center}
\end{figure}

We now turn to some 3D simulations, again with the standard diffuse-interface approximation (\ref{eq:Wf-d}).
The natural analogue of the two disks initial condition in three dimensions is two disjoint spheres.
However, unlike disks in 2D, spheres in 3D are stationary under Willmore flow, and would in fact shrink to naught in the presence of even the slightest additional penalty on perimeter (i.e. when $\gamma>0$).
We therefore add an expansionary bulk energy term:
\begin{equation}
\mathcal{W}_\varepsilon (u) + \gamma \mathcal{H}_\varepsilon (u) - \alpha \int u \, dx.
\end{equation}
with $\alpha>0$.
Schemes (\ref{eq:findiffdegiorgi}) and (\ref{eq:semiimplicit}), which correspond to the $\alpha=0$ case, adapt trivially to the $\alpha\not=0$ case.

Figure \ref{fig:expand2d} shows that the inclusion of the expansion term in 2D does not alter the formation of a cross (and failure to merge and become regular) for the two disk initial data.
This simulation was carried out with the same parameters as before and $\alpha = 15$.
\begin{figure}[h]
\begin{center}
\includegraphics[width=0.24\textwidth]{figures/findiff/standard1}
\includegraphics[width=0.24\textwidth]{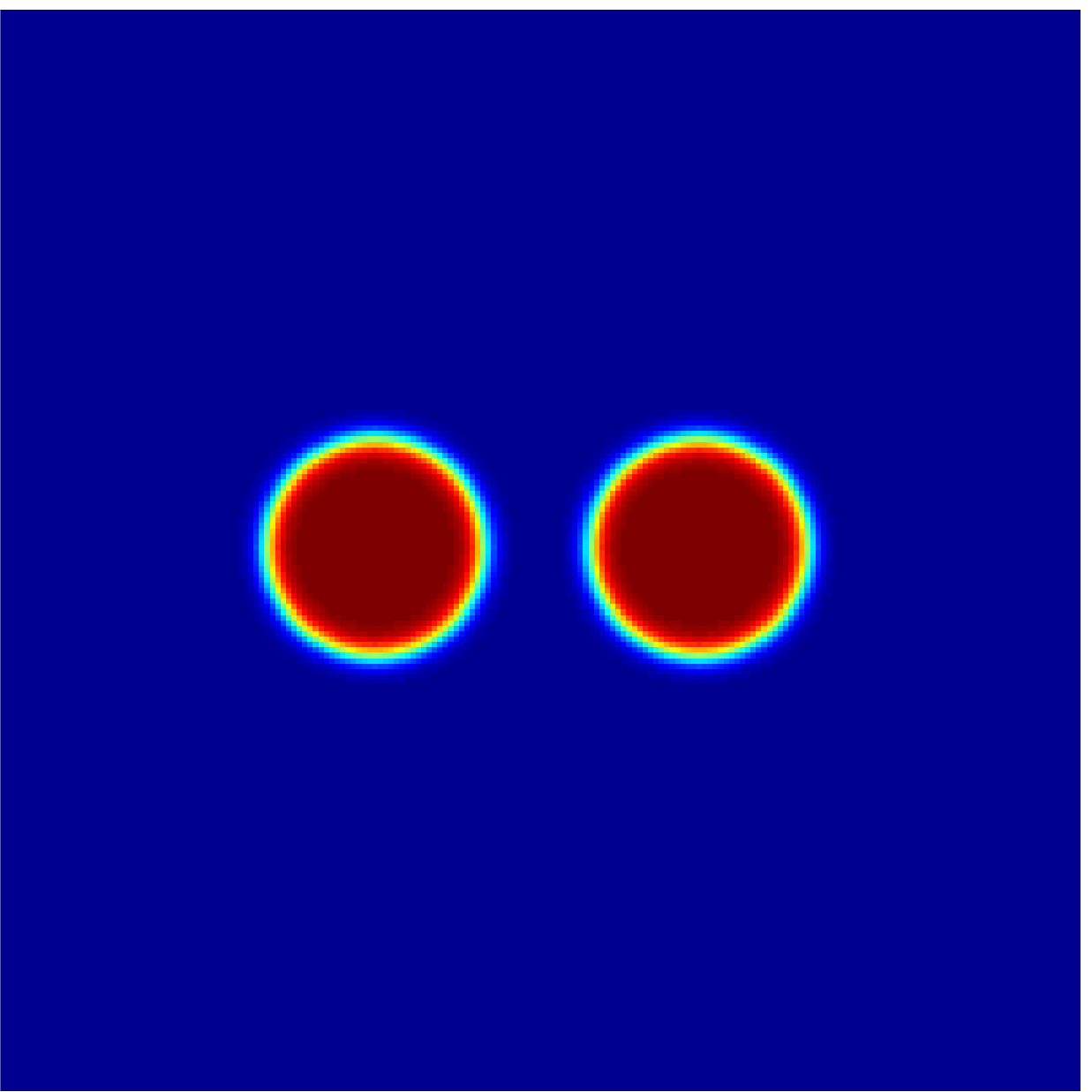}
\includegraphics[width=0.24\textwidth]{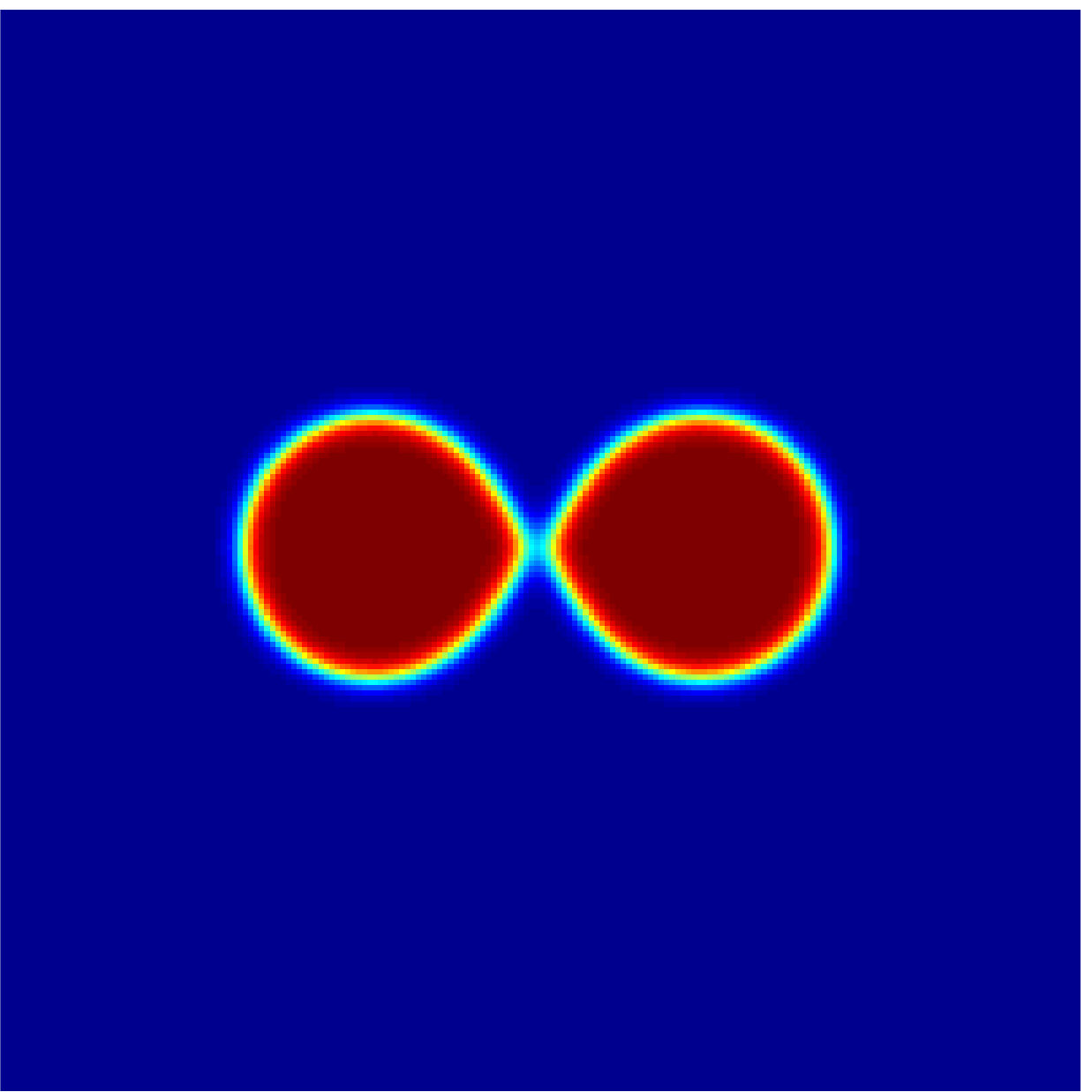}
\includegraphics[width=0.24\textwidth]{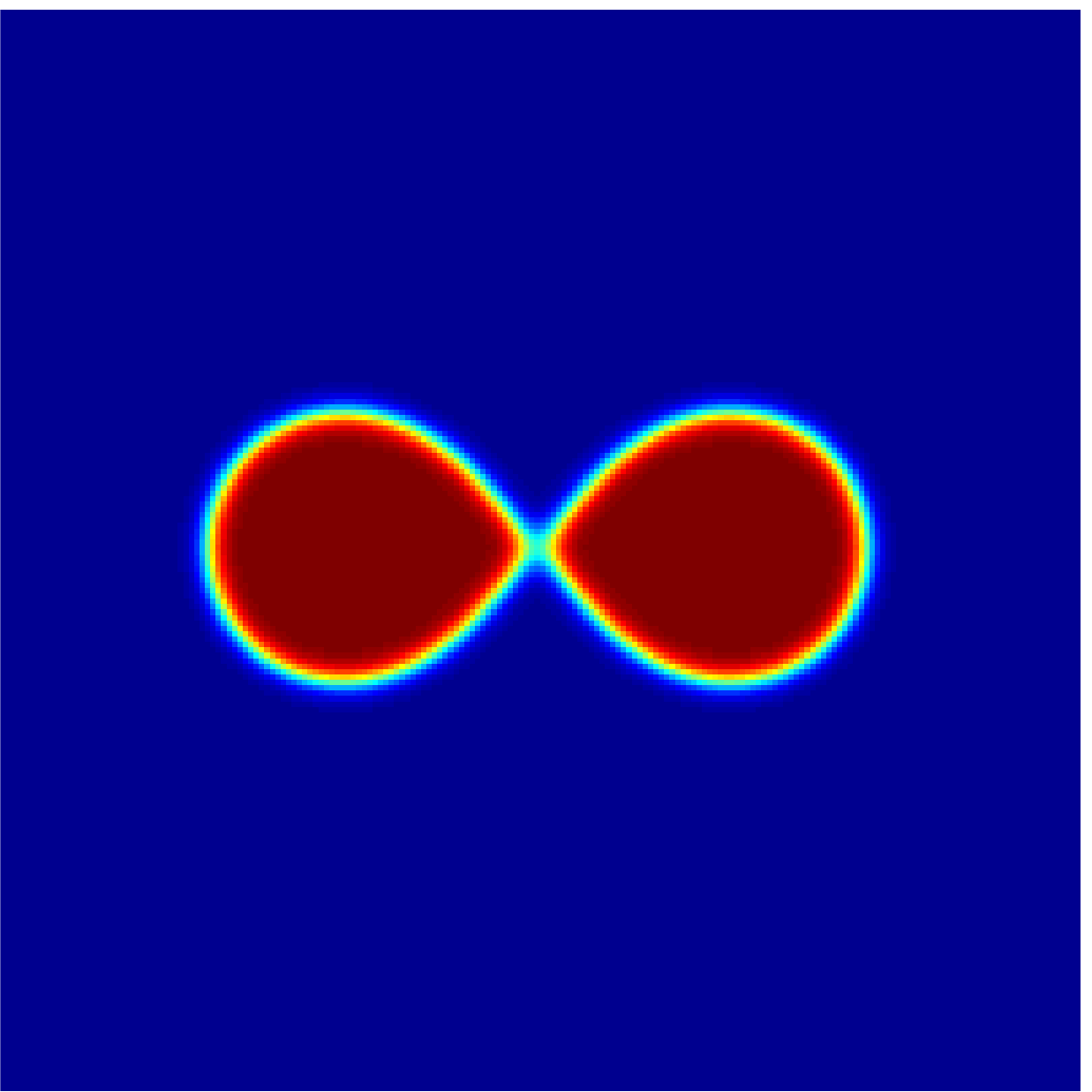}
\caption{\footnotesize Willmore flow with perimeter penalty and volumetric expansion term, simulated using a finite differences discretization of the standard diffuse-interface approximation.}
\label{fig:expand2d}
\end{center}
\end{figure}

Figure \ref{fig:expand3d} shows simulations in 3D with two disjoint spheres of equal size as initial data.
The parameters were $\varepsilon = 0.06$, $\gamma=\frac{1}{4}$, and $\alpha = 15$.
The spatial resolution was $100\times 100\times 100$.
The volumetric term leads to expansion of the spheres, which eventually touch.
Unlike the 2D situation, the two surfaces merge and become instantaneously regular.

\begin{figure}[h]
\begin{center}
\includegraphics[width=0.24\textwidth]{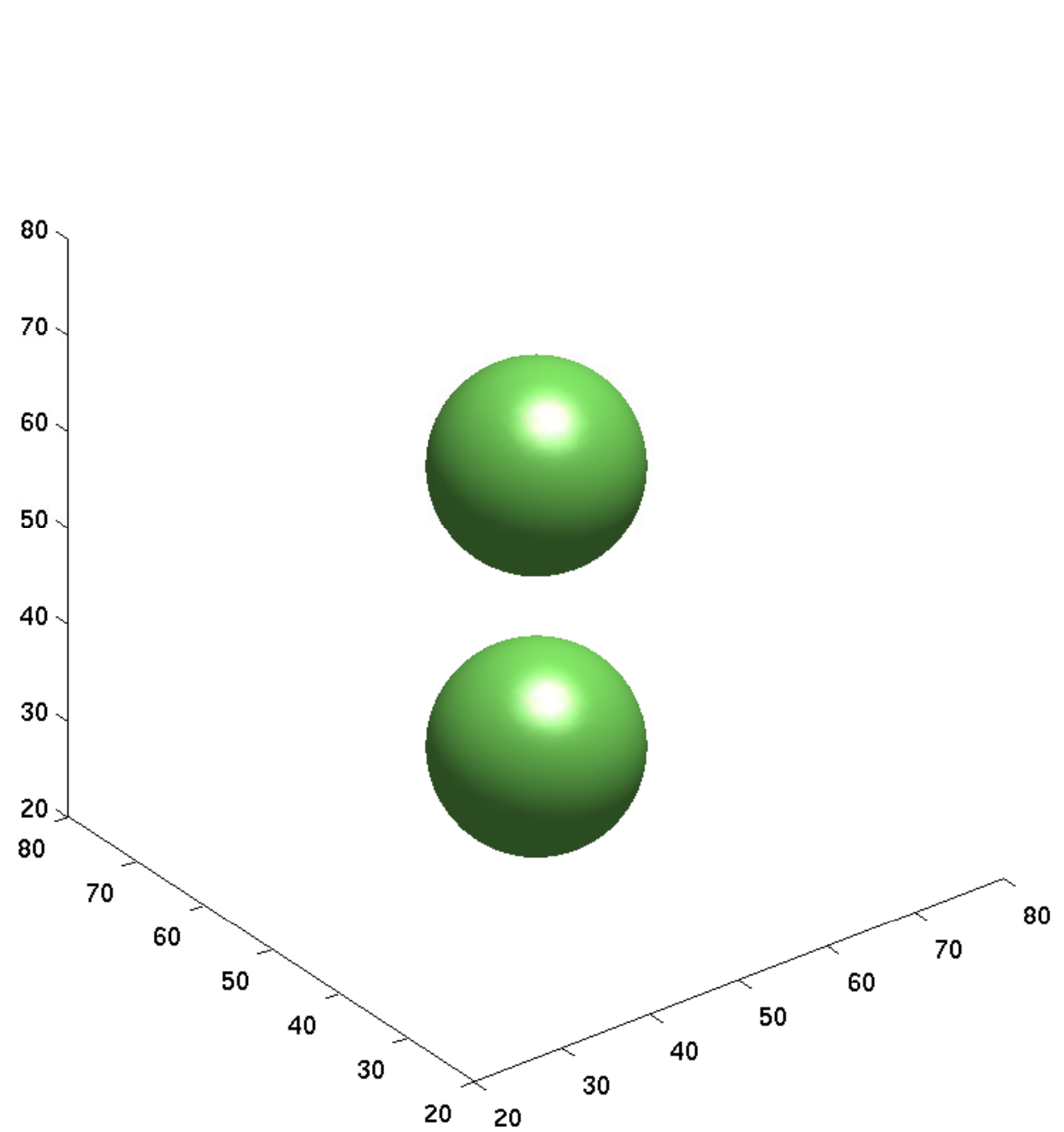}
\includegraphics[width=0.24\textwidth]{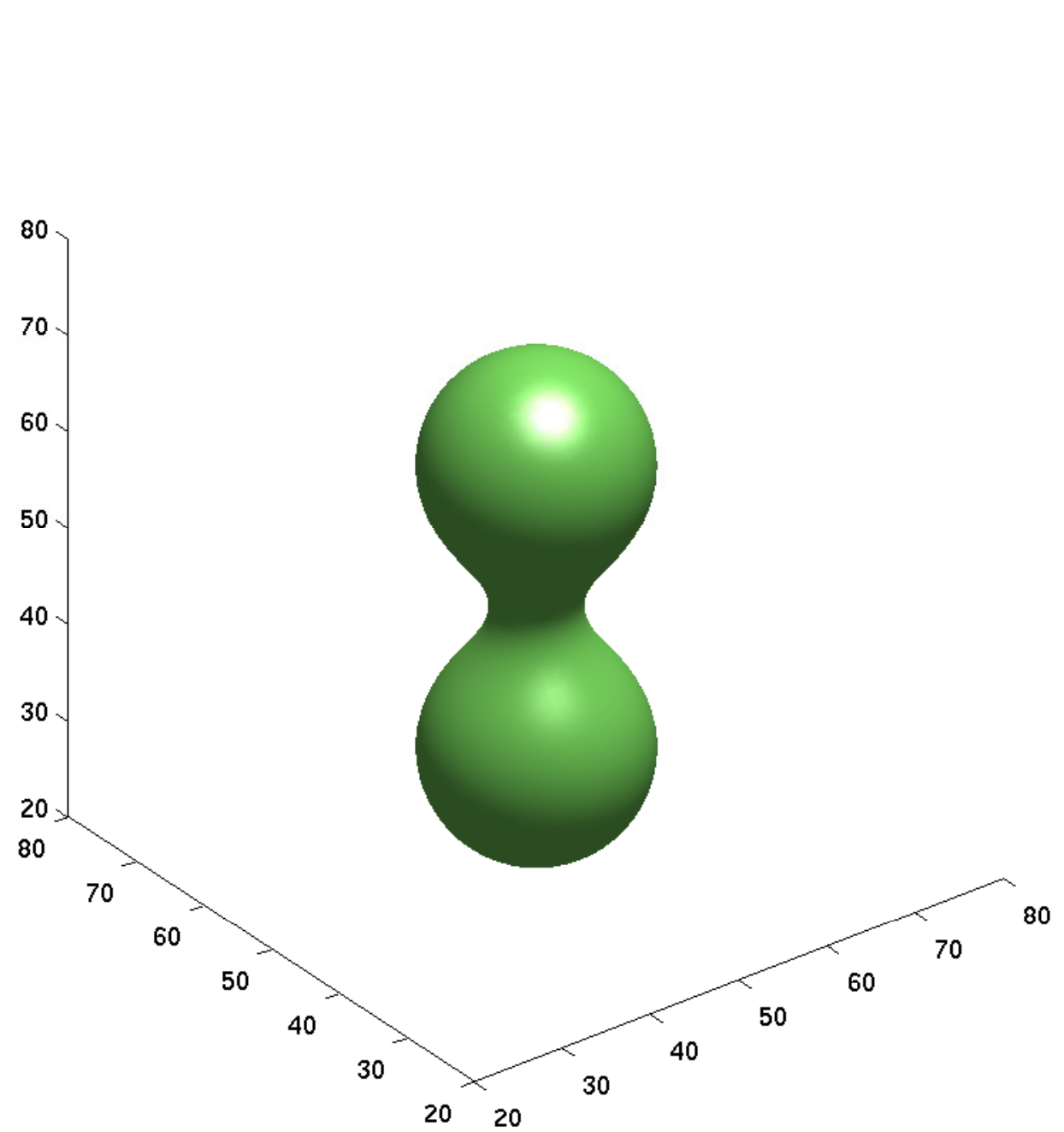}
\includegraphics[width=0.24\textwidth]{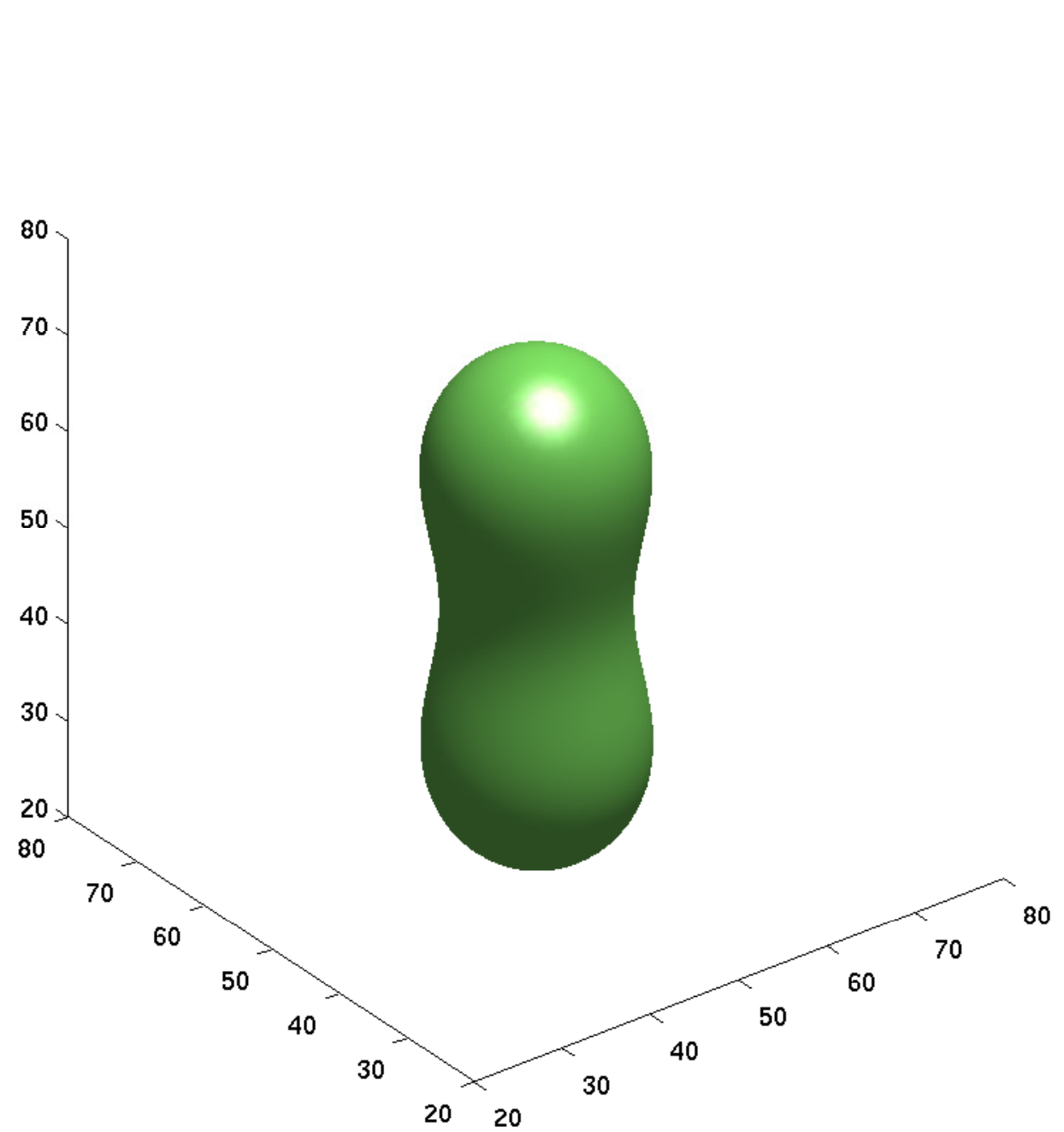}
\includegraphics[width=0.24\textwidth]{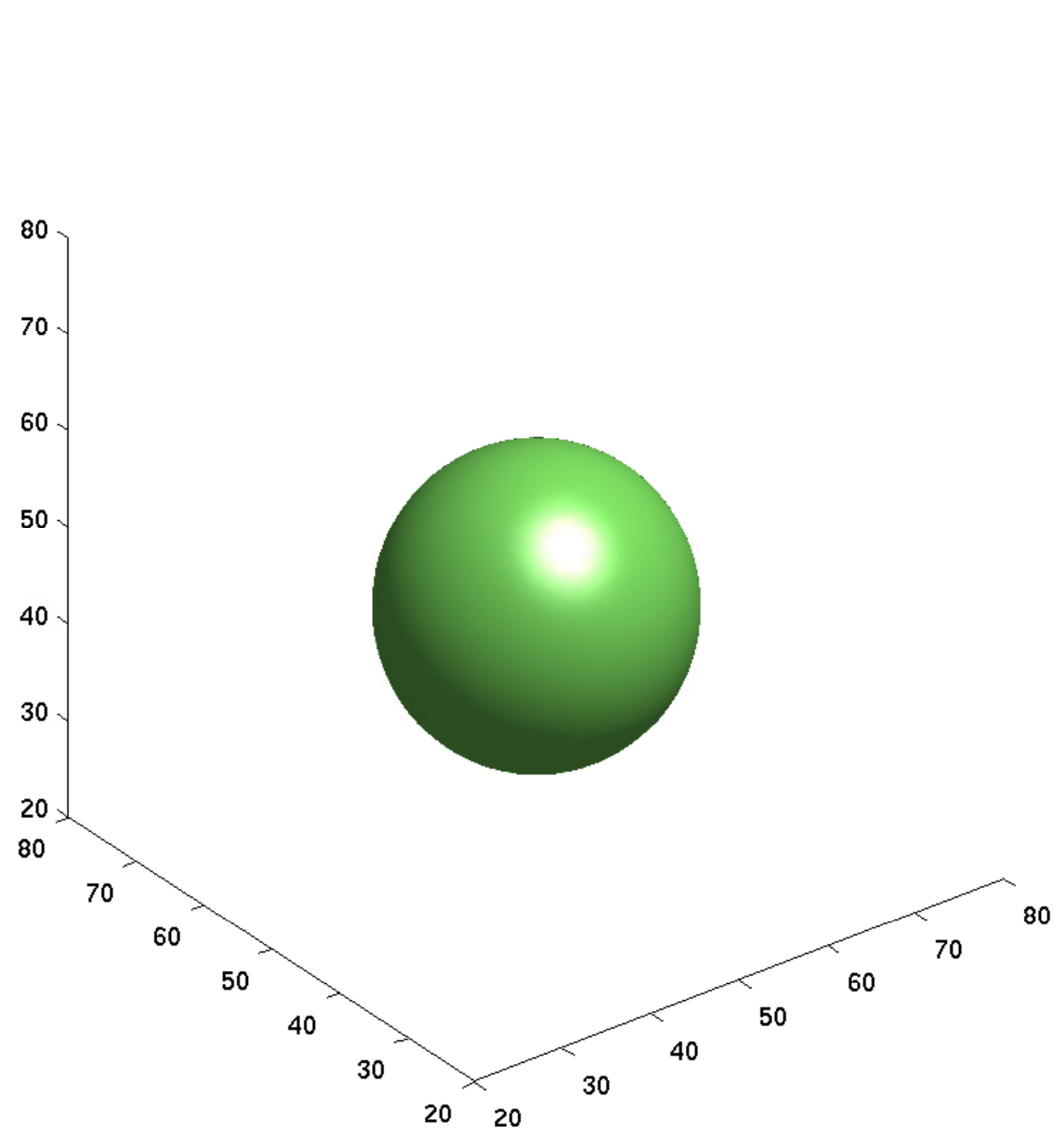}
\caption{\footnotesize 3D simulations with the standard diffuse-interface approximation of Willmore flow, together with a volumetric expansion term.}
\label{fig:expand3d}
\end{center}
\end{figure}

\section{Numerical simulations with Bellettini's approximation}
\label{sec:bellettini_numerics}
In this section, we provide numerical experiments in the plane with Bellettini's Gamma-convergent approximation (\ref{eq:bell}), and investigate what happens at topological changes.
In the interest of isolating potential numerical artifacts from inherent behavior of Willmore flow that results from this approximation, we keep the numerical scheme as simple as possible: It is a straight-forward finite differences discretization, with explicit time stepping.
Unsurprisingly, computations with this scheme are very slow, owing to the extremely stringent stability restriction on the time step size.
However, they appear to be also very robust.
In particular, a discrete form of the energy is observed to decrease at every time step.
Due to the delicate nature of the question (topological changes in a {\em fourth} order gradient flow for a curvature dependent functional!), we give full details of the implementation.

We will work with the following regularized version of (\ref{eq:bell}):
\begin{equation}
\label{eq:regbell}
\mathcal{W}_{\varepsilon,\delta}(u) := \frac{1}{2}\int \left\{ \left( \nabla \cdot \frac{\nabla u}{\sqrt{|\nabla u|^2+\delta}} \right)^2 + \gamma \right\} \left( \frac{\varepsilon}{2} |\nabla u|^2 + \frac{1}{\varepsilon} W(u) \right) \, dx.
\end{equation}
Here $\gamma>0$ is added to ensure a uniform bound for $\mathcal{H}_\eps(u_\eps)$, necessary for Gamma-convergence to the $L^1$-relaxation of the Willmore functional.
For simplicity, our exposition is restricted to $\mathbb{R}^2$ below; extension to arbitrary dimensions is straight forward.
To begin with, the $L^2$ gradient flow for (\ref{eq:regbell}) leads to the following evolution:
\begin{equation}
\begin{split}
 \partial_t u =& -\frac{\partial}{\partial x} \left( \frac{ (u_y^2+\delta) \partial_x(\kappa_\delta {h}_\varepsilon) - u_x u_y \partial_y(\kappa_\delta {h}_\varepsilon) }{ (|\nabla u|^2 + \delta)^{\frac{3}{2}} } \right)\\
& -\frac{\partial}{\partial y} \left( \frac{ (u_x^2+\delta) \partial_y(\kappa_\delta {h}_\varepsilon) - u_x u_y \partial_x(\kappa_\delta {h}_\varepsilon) }{ (|\nabla u|^2 + \delta)^{\frac{3}{2}} } \right)\\
& + \frac{\varepsilon}{2} \nabla \cdot \Big( (\kappa_\delta^2+\gamma) \nabla u \Big) - \frac{1}{2\varepsilon} (\kappa_\delta^2+\gamma)W'(u).
\end{split}
\end{equation}
where
\begin{equation}
\kappa_\delta := \nabla \cdot \left( \frac{\nabla u}{\sqrt{|\nabla u|^2 + \delta}} \right)
\end{equation}
and ${h}_\varepsilon$ denotes the diffuse surface area density, ${h}_\varepsilon(u)= \frac{\eps}{2}|\nabla u|^2 + \eps^{-1}W(u)$.

Working on a uniform spatial grid with periodic boundary conditions, let $D^+$ and $D^-$ denote the standard forward and backward difference quotients in the direction of their subscript.
Let $u^n$ denote the solution at the $n$-th time step.
Let's start with the discretization of the diffuse surface area density  ${h}_\varepsilon(u)$:
\begin{equation}
M(u) = \frac{\varepsilon}{4} \left( \big(D^+_x u\big)^2 + \big(D^-_x u\big)^2 + \big(D^+_y u\big)^2 + \big(D^-_y u\big)^2 \right) + \frac{1}{\varepsilon} W(u). 
\end{equation}
The denominator in the curvature term $\kappa_\delta$ can be discretized in any one of the following four ways (we'll use all):
\begin{equation}
\label{eq:Du}
Du^{\pm,\pm} = \sqrt{ (D^\pm_x u)^2 + (D^\pm_y u)^2 + \delta }
\end{equation}
where the first and second $\pm$ in the superscript of $Du$ refer to the signs of difference quotients in the $x$ and the $y$ directions, respectively.
The superscript will be dropped for convenience below, whenever it is just $\pm,\pm$.
The choice of sign for the difference quotient in each coordinate direction in (\ref{eq:Du}) determines the signs of all subsequent difference quotients, as indicated with $\pm$ or $\mp$ signs below. 
Approximation to the curvature term $\kappa_\delta$ can be obtained as:
\begin{equation}
K^{\pm,\pm}(u) = D^\mp_x \left( \frac{D^\pm_x u}{Du} \right) + D^\mp_y \left( \frac{D^\pm_y u}{Du} \right)
\end{equation}
where the superscripts of $K$ indicate the signs for the difference quotients in the $x$ and $y$ directions, respectively.

Next, define
\begin{equation}
  \begin{split}
    A_1^{\pm,\pm}(u) &:= D^\mp_x \left( \frac{ \left[ D^\pm_x (K(u)M(u))
        \right] \left[ (D^\pm_y u)^2 + \delta \right] }{(Du)^3} \right)\\
    & \quad - D^\mp_y \left( \frac{ \left[ D^\pm_x (KM) \right] (D^\pm_x u) (D^\pm_y u) }{(Du)^3} \right),\\
    A_2^{\pm,\pm}(u) &:= D^\mp_y \left( \frac{ \left[ D^\pm_y (K(u)M(u))
        \right] \left[ (D^\pm_x u)^2 + \delta \right] }{(Du)^3} \right)\\
    & \quad - D^\mp_x \left( \frac{ \left[ D^\pm_y (K(u)M(u)) \right] (D^\pm_x u) (D^\pm_y u) }{(Du)^3} \right).
  \end{split}
\end{equation}
where superscripts of $A_1$ and $A_2$ indicate once again the chosen signs for the difference quotients in the $x$ and $y$ directions, in that order.
Let
\begin{equation}
\begin{split}
A_1^*(u) &:= A_1^{+,+}(u) + A_1^{+,-}(u) + A_1^{-,+}(u) + A_1^{-,-}(u),\\
A_2^*(u) &:= A_2^{+,+}(u) + A_2^{+,-}(u) + A_2^{-,+}(u) + A_2^{-,-}(u).
\end{split}
\end{equation}
Define also the discrete squared curvature:
\begin{equation}
K_*^2(u) := \frac{1}{4} \left( (K^{+,+}(u))^2 + (K^{+,-}(u))^2 + (K^{-,+}(u))^2 + (K^{-,-}(u))^2 \right).
\end{equation}
Let
\begin{equation}
\begin{split}
A_3(u) :=& - \frac{\varepsilon}{2} \Big\{ D^-_x \big( (K_*^2+\gamma) D^+_x u^n \big) + D^-_y \big( (K_*^2+\gamma) D^+_y u^n \big) +\\
& D^+_x \big( (K_*^2+\gamma) D^-_x u^n \big) + D^+_y \big( (K_*^2+\gamma) D^-_y u^n \big) \Big\}\\
&+ \frac{1}{\varepsilon} \big( K_*^2 + \gamma \big) W'(u^n).
\end{split}
\end{equation}
Finally, our update scheme is:
\begin{equation}
\frac{u^{n+1}-u^n}{\delta t} = - \Big(  A_1(u^n) +  A_2(u^n) + \frac{1}{2}A_3(u^n) \Big).
\end{equation}
With time step size $\delta t>0$ chosen small enough compared to the spatial grid size, this scheme is guaranteed to decrease the following discrete form of the energy
\begin{equation}
\label{eq:discrete_bellettini}
\sum_{i,j} \big( K_*^2(u^n_{i,j}) + \gamma \big) M(u^n_{i,j})
\end{equation}
as can be easily verified by differentiating (\ref{eq:discrete_bellettini}) with respect to time, and summing by parts a few times.
Of course, it must be mentioned that the convergence of energy (\ref{eq:discrete_bellettini}) to, say, (\ref{eq:regbell}) as the grid size and the regularization parameter $\delta$ are appropriately sent to $0$ has not been established rigorously; we merely give some numerical evidence.

In 2D, the computational domain was $[0,1]^2$ with a spatial resolution of $100\times 100$.
Periodic boundary conditions were used.
The parameters were chosen to be $\varepsilon = 0.045$, $\gamma = \frac{1}{2}$, and $\delta = 0.01$.
As an initial guess, $u^0$ was taken to be the characteristic function of the union of two disjoint disks.
During the evolution, the disks are observed to initially expand as disks, as expected.
However, once they are within a small distance (related to the diffuse-interface thickness $\varepsilon$) of each other, they are unable to get any closer.
Topological changes appear to be precluded; in particular, neither a merger subsequently leading to a smooth evolution occurs (as was previously seen in various numerical implementations of Willmore flow via implicit representations), nor a corner forms as with the De Giorgi approximation.
Instead, the curves continue to expand, but are no longer circles.

\begin{figure}[h]
\begin{center}
\includegraphics[width=0.24\textwidth]{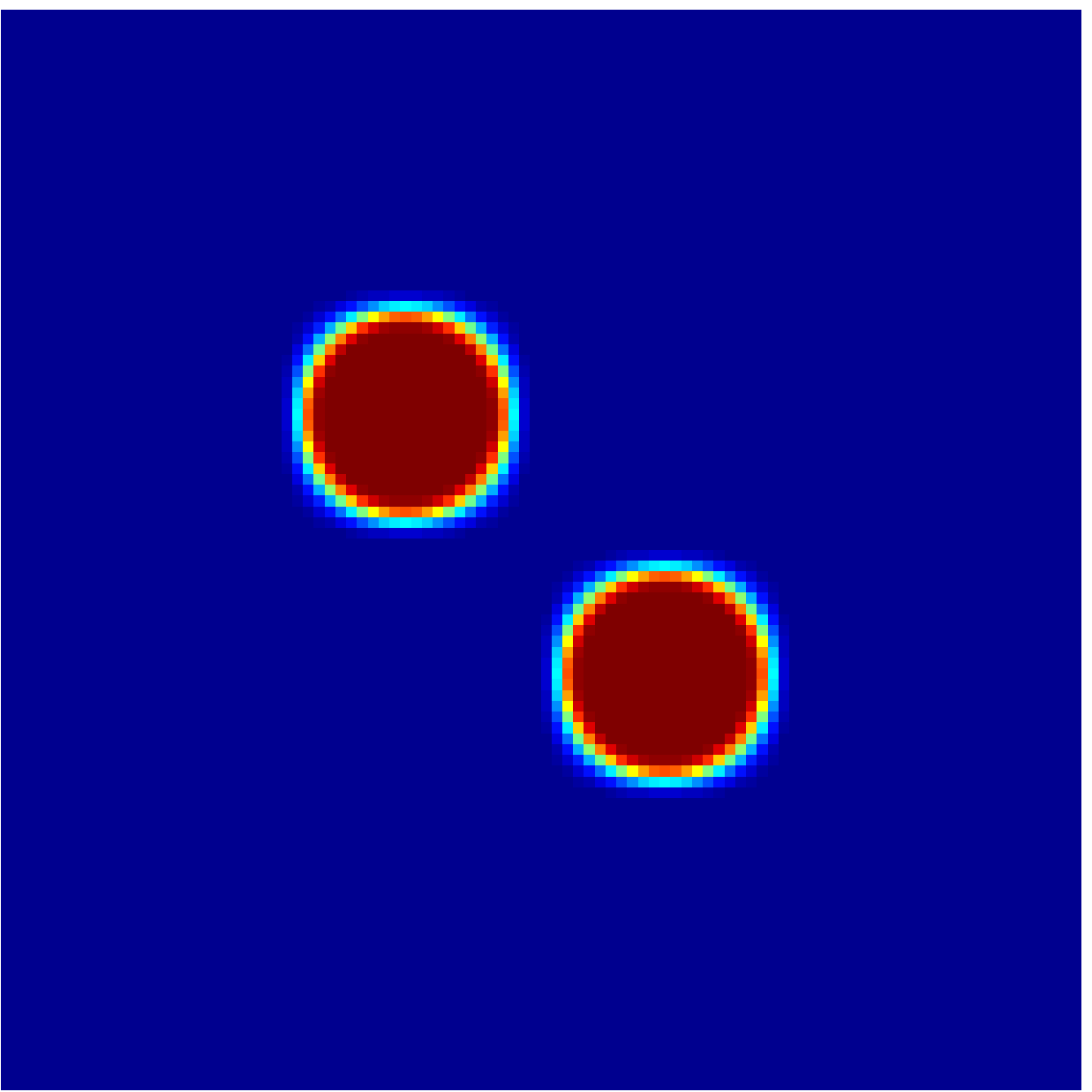}
\includegraphics[width=0.24\textwidth]{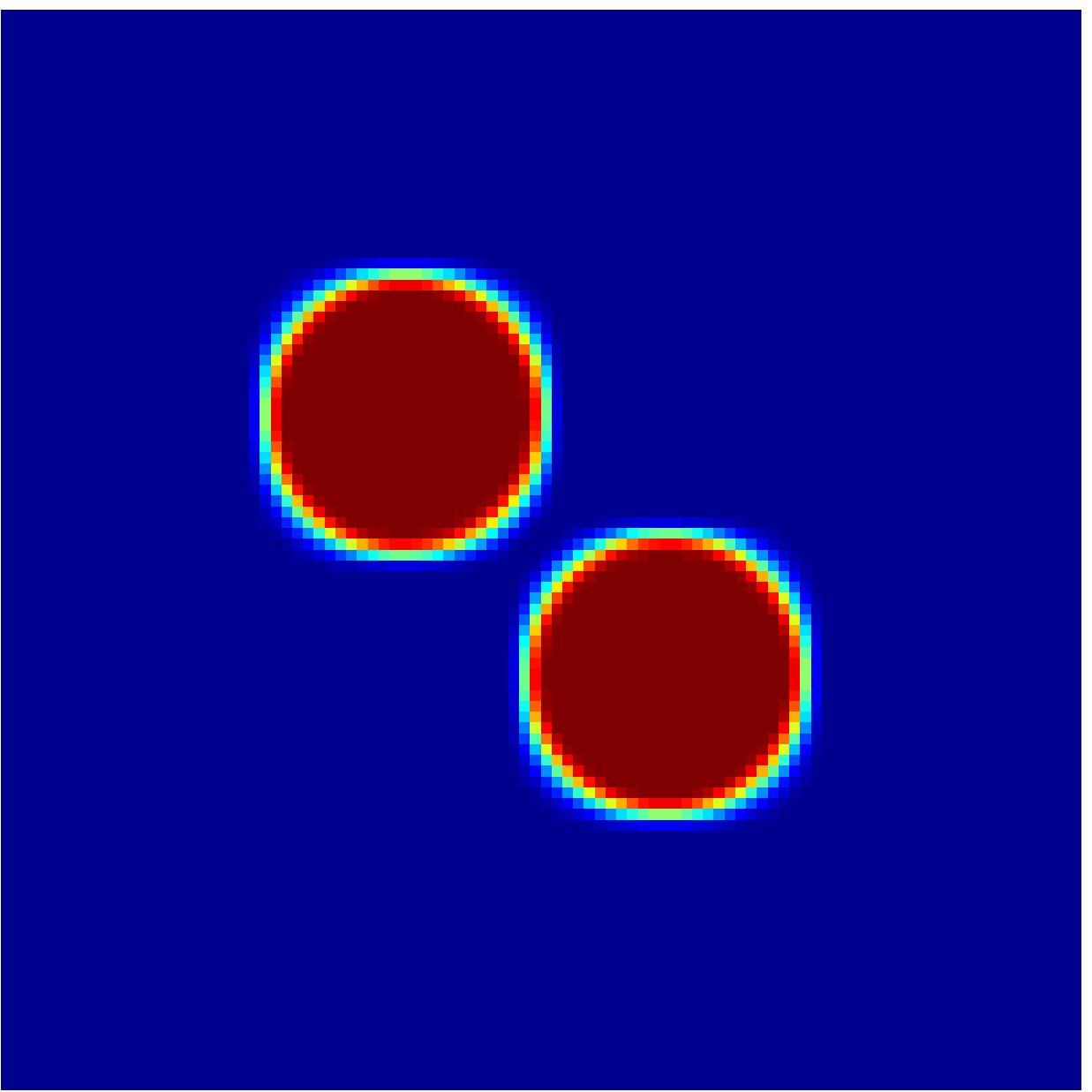}
\includegraphics[width=0.24\textwidth]{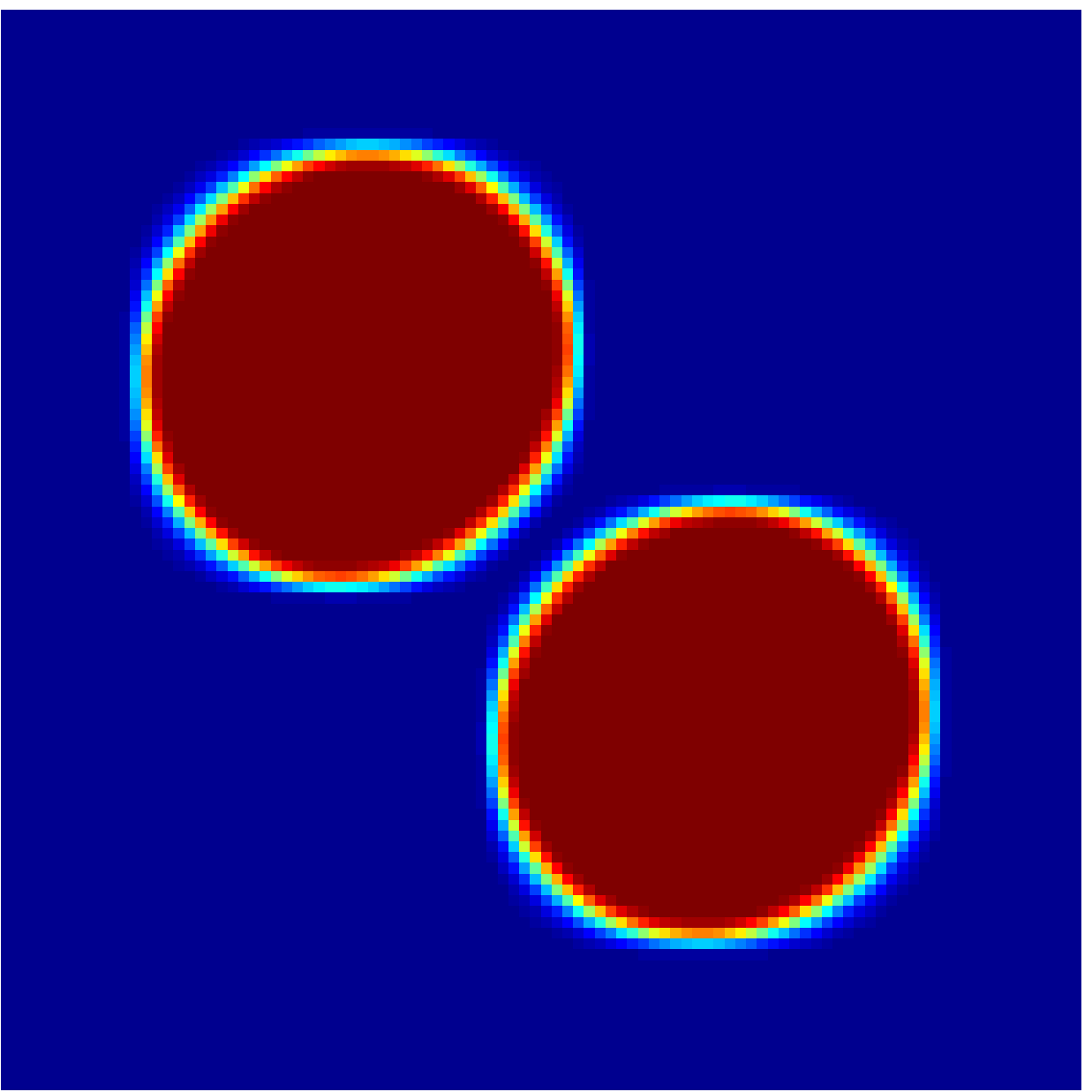}
\includegraphics[width=0.24\textwidth]{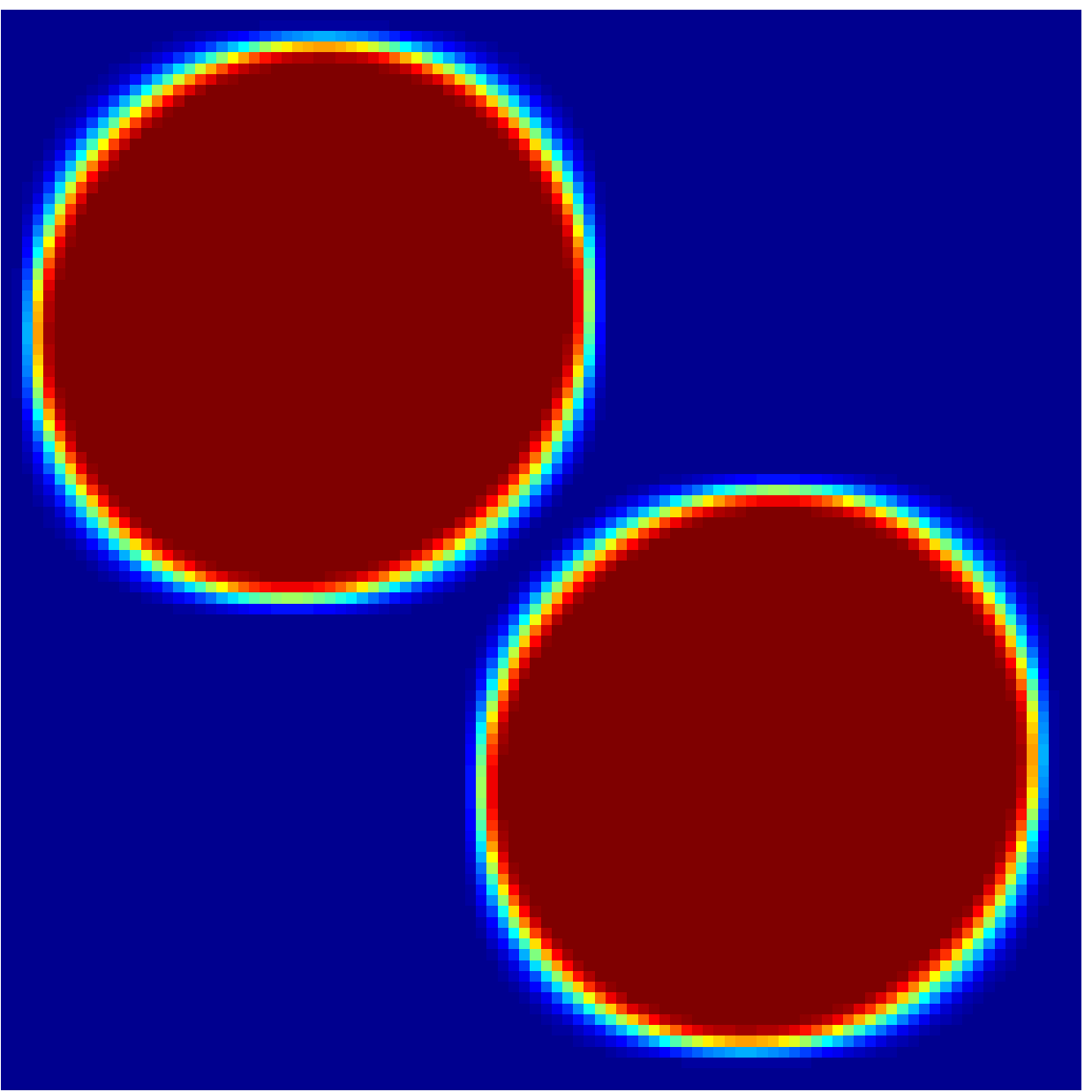}
\caption{\footnotesize Evolution of disks under $L^2$ gradient flow (in the bulk) for Bellettini's approximation to the Willmore energy.
Topological changes appear to be precluded: the disks do not merge.
After getting within a small distance of each other (related to diffuse-interface thickness), they continue their expansion, but are no longer circles.}
\label{fig:bellettini_disks}
\end{center}
\end{figure}

Next, we explore what happens in 3D.
Once again, as in Section \ref{sec:finitedifferencesimplementation}, we include a volumetric expansion term so that spheres would grow.
Figure \ref{fig:cylinders} shows computations with two parallel cylinders as initial condition, which in fact corresponds to a 2D computation; the difference from the experiment of Figure \ref{fig:bellettini_disks} is the inclusion of the volumetric expansion term.
We see that even in the presence of the additional driving force bringing the interfaces into collision, topological change is precluded, just as in the 2D experiment of Figure \ref{fig:bellettini_disks}.

\begin{figure}[h]
\begin{center}
\includegraphics[width=0.24\textwidth]{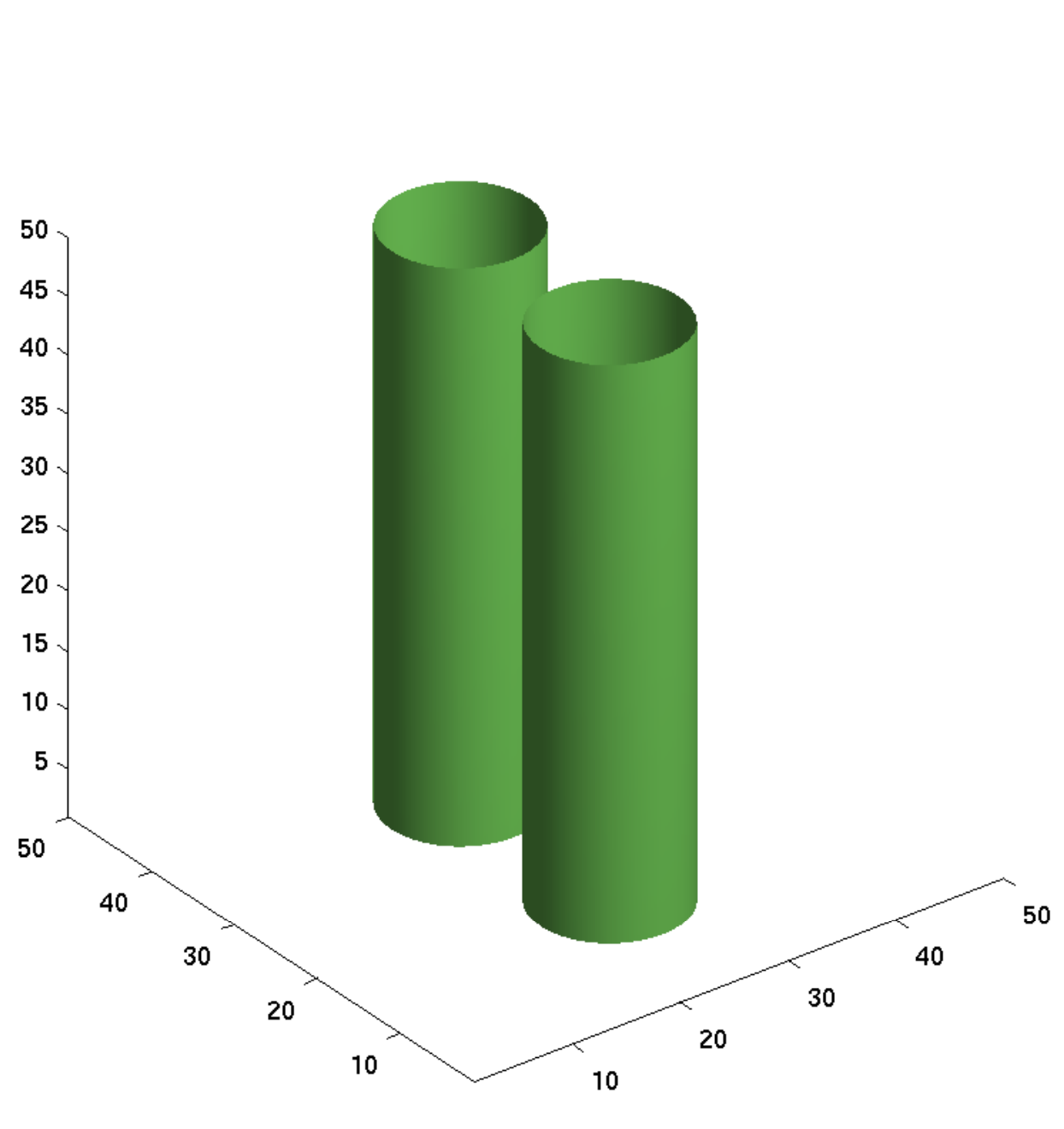}
\includegraphics[width=0.24\textwidth]{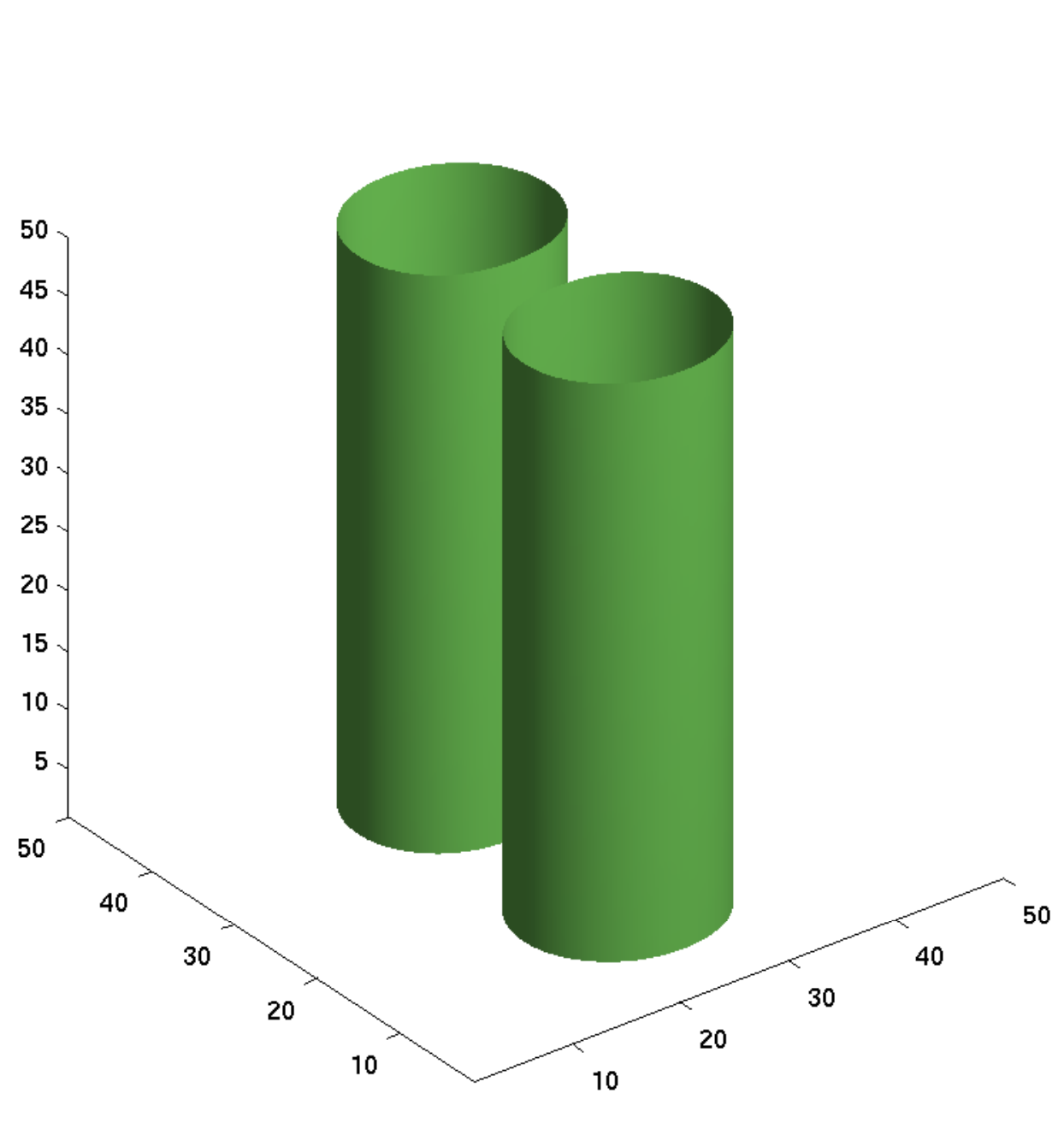}
\includegraphics[width=0.24\textwidth]{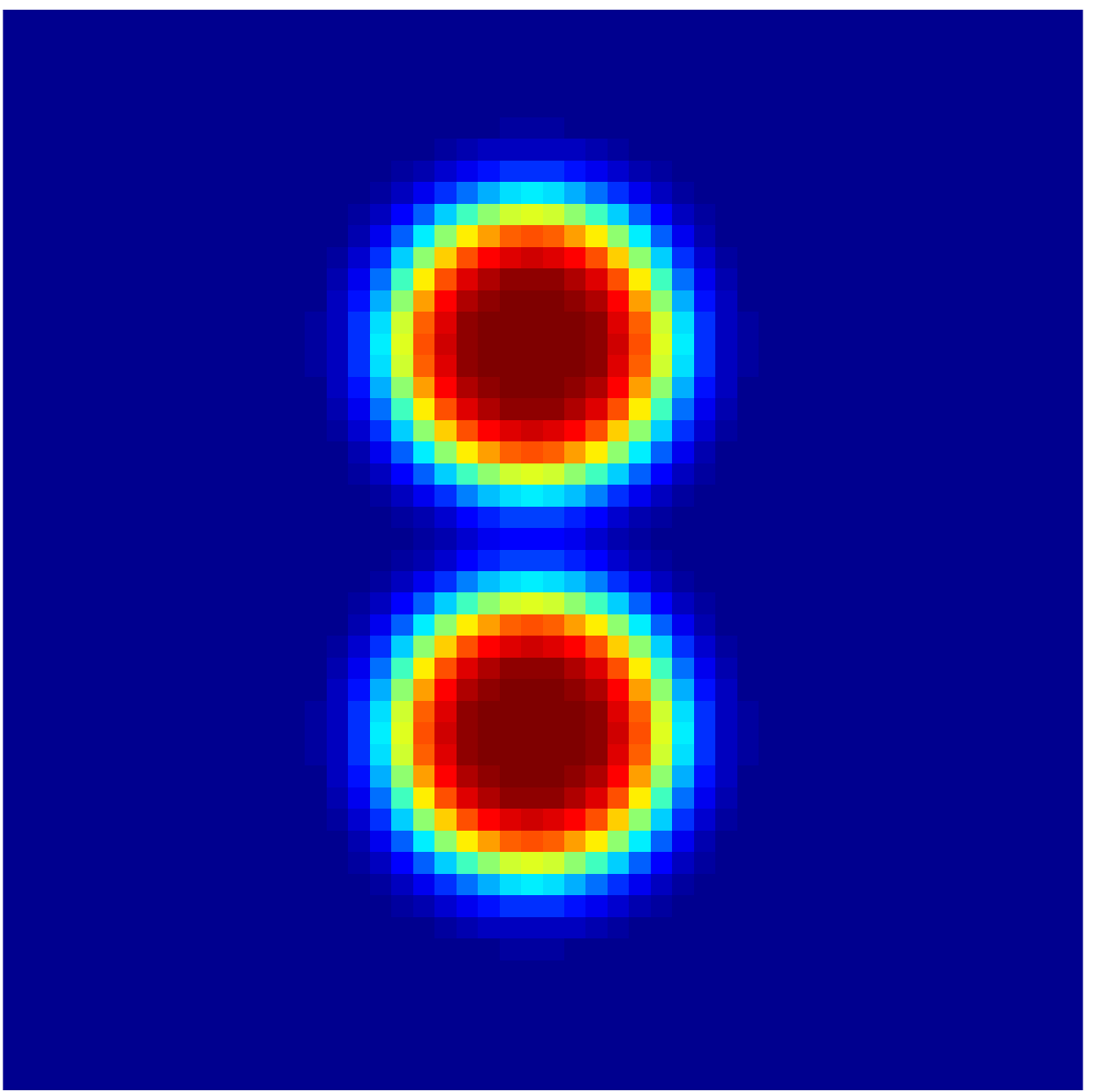}
\includegraphics[width=0.24\textwidth]{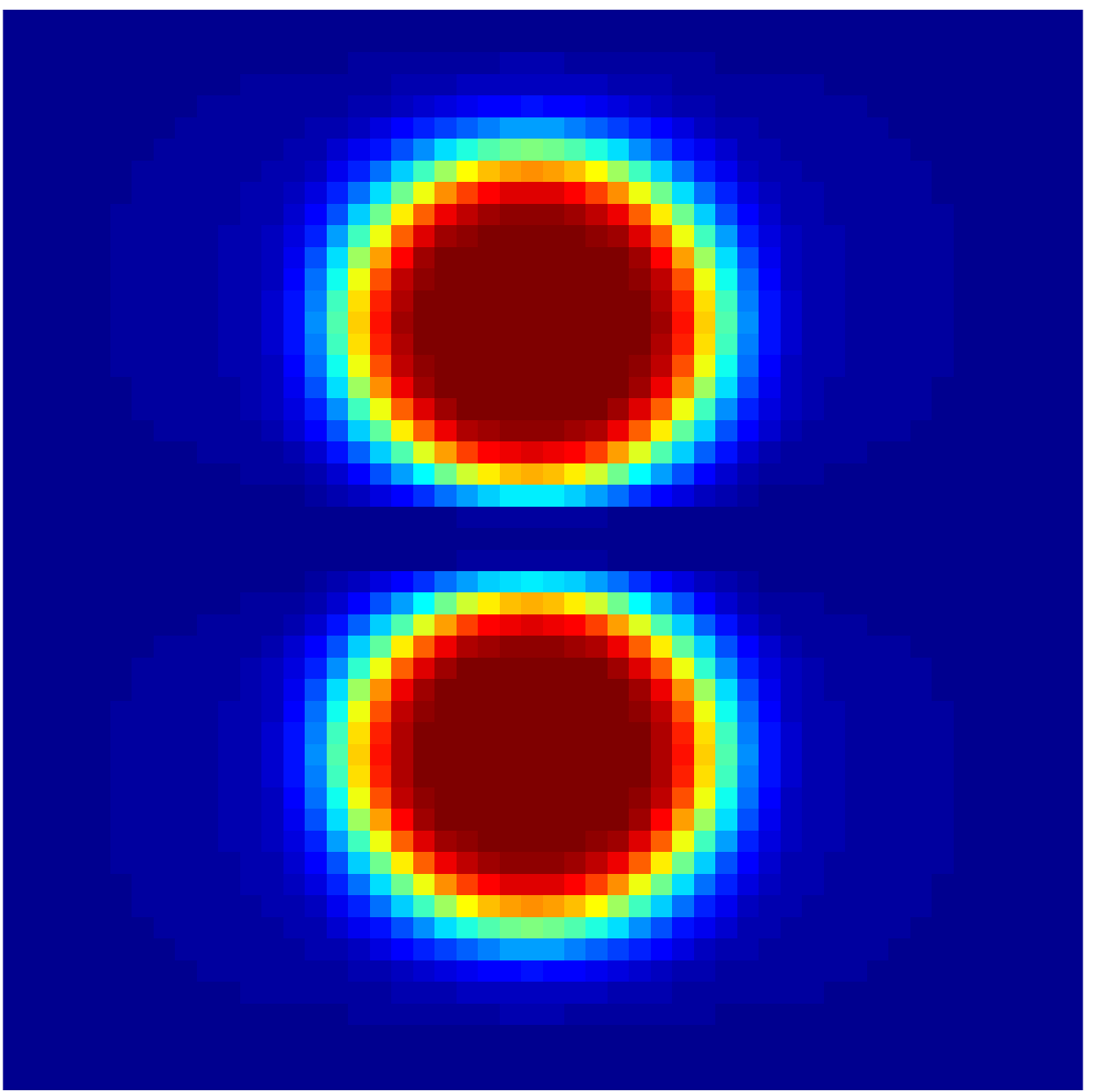}
\caption{\footnotesize Evolution of cylinders under $L^2$ gradient flow for Bellettini's approximation of Willmore energy, with a bulk energy term encouraging expansion added.
The evolution is essentially 2D; the only real difference from the simulation of Figure \ref{fig:bellettini_disks} is the inclusion of the bulk energy term.
Just like in the no bulk energy term case of Figure \ref{fig:bellettini_disks}, the cylinders cannot merge even though they come into contact, despite the additional driving force that slams interfaces into each other.}
\label{fig:cylinders}
\end{center}
\end{figure}

On the other hand, Figure \ref{fig:bellettini_spheres} shows simulations with two disjoint spheres as the initial condition.
The computational domain is $[0,1]^3$ and the spatial resolution is $50\times 50\times 50$.
The parameters were chosen to be $\varepsilon = 0.12$, $\gamma = \frac{1}{4}$, and $\delta = 0.01$.
In this case, the spheres merge once they come into contact (i.e. "feel" each other due to the diffuse-interface thickness).
See the remark in Section \ref{sec:discussion} for an explanation of why this topological change is not precluded.

\begin{figure}[h]
\begin{center}
\includegraphics[width=0.24\textwidth]{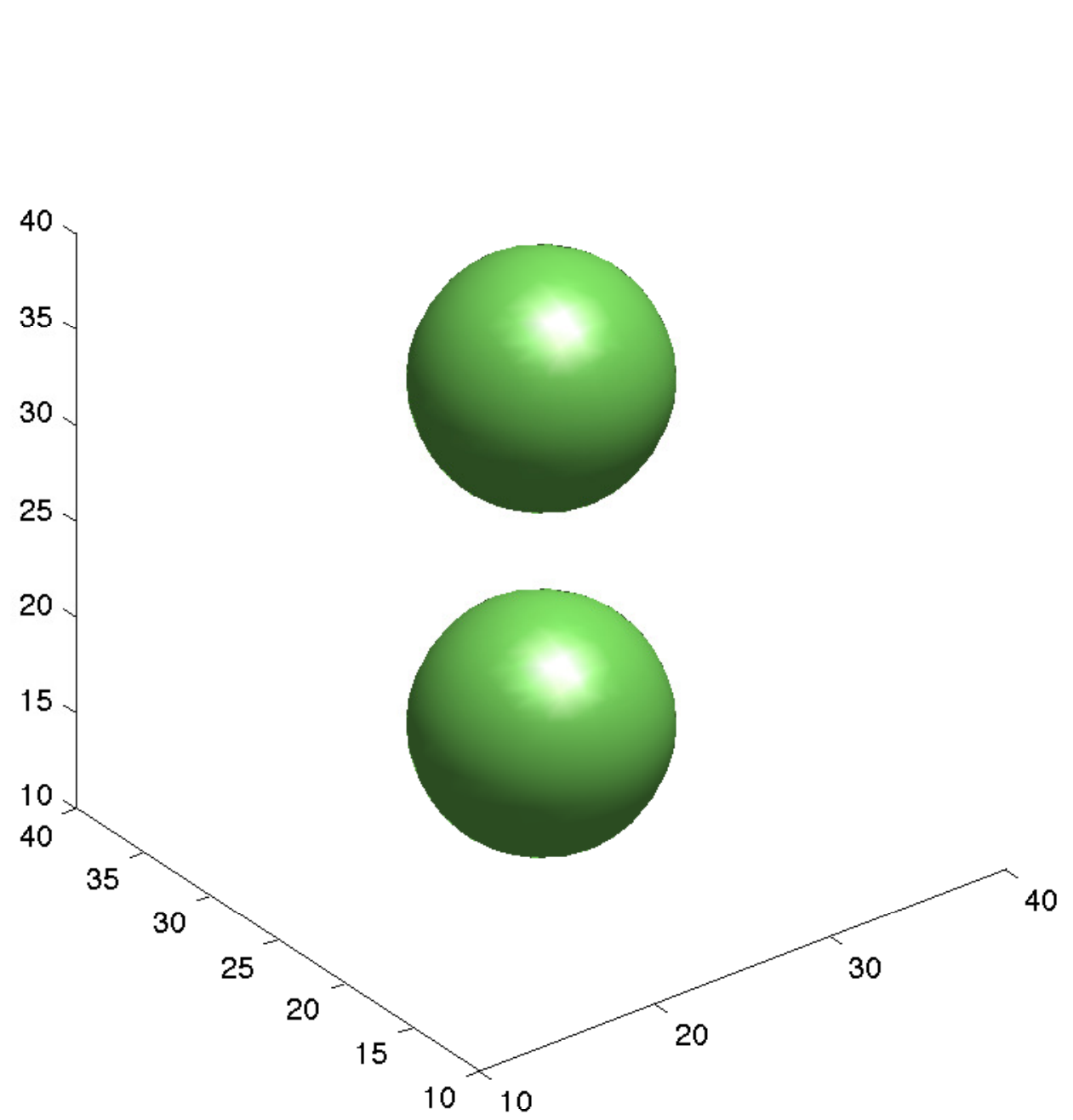}
\includegraphics[width=0.24\textwidth]{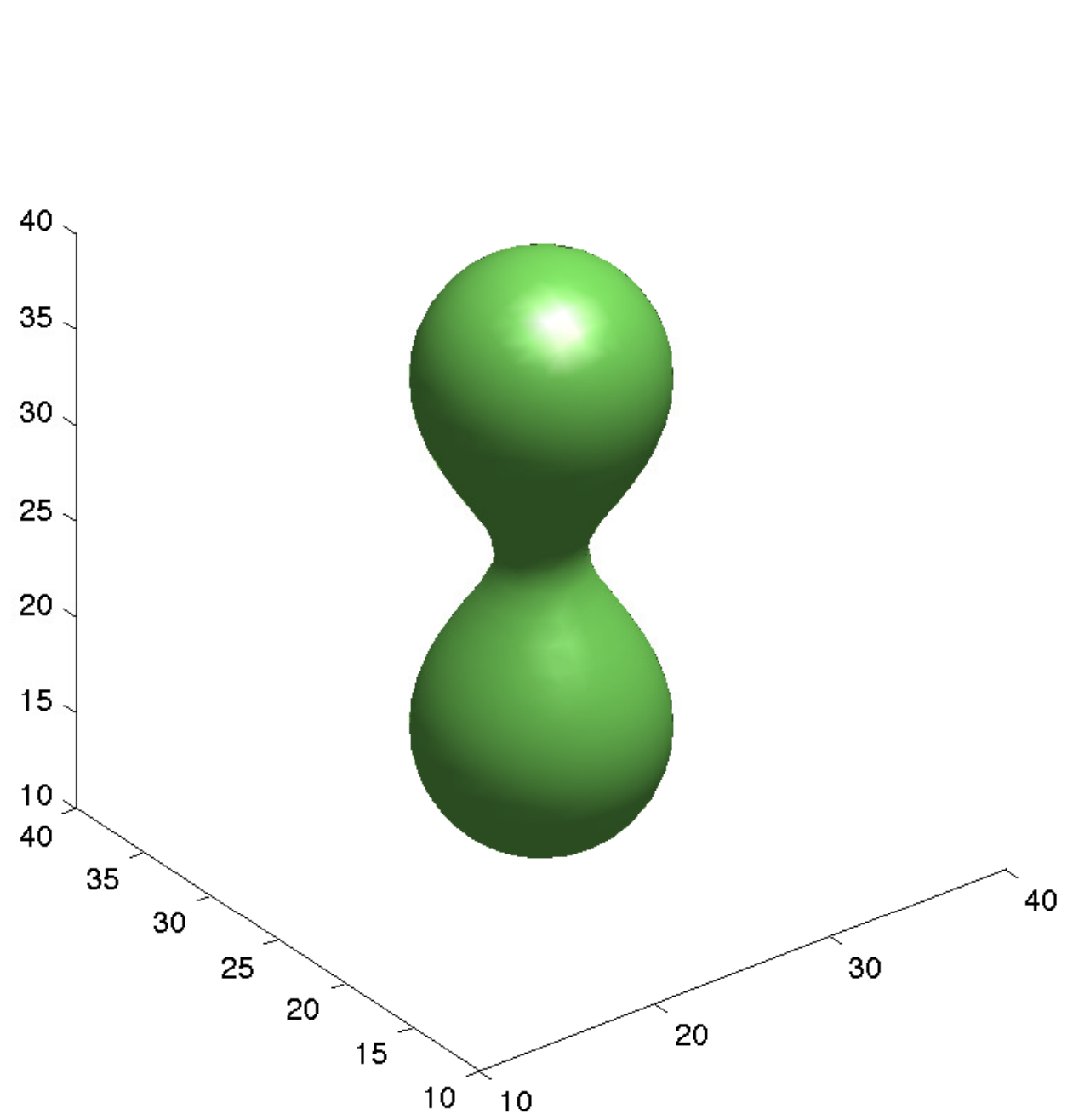}
\includegraphics[width=0.24\textwidth]{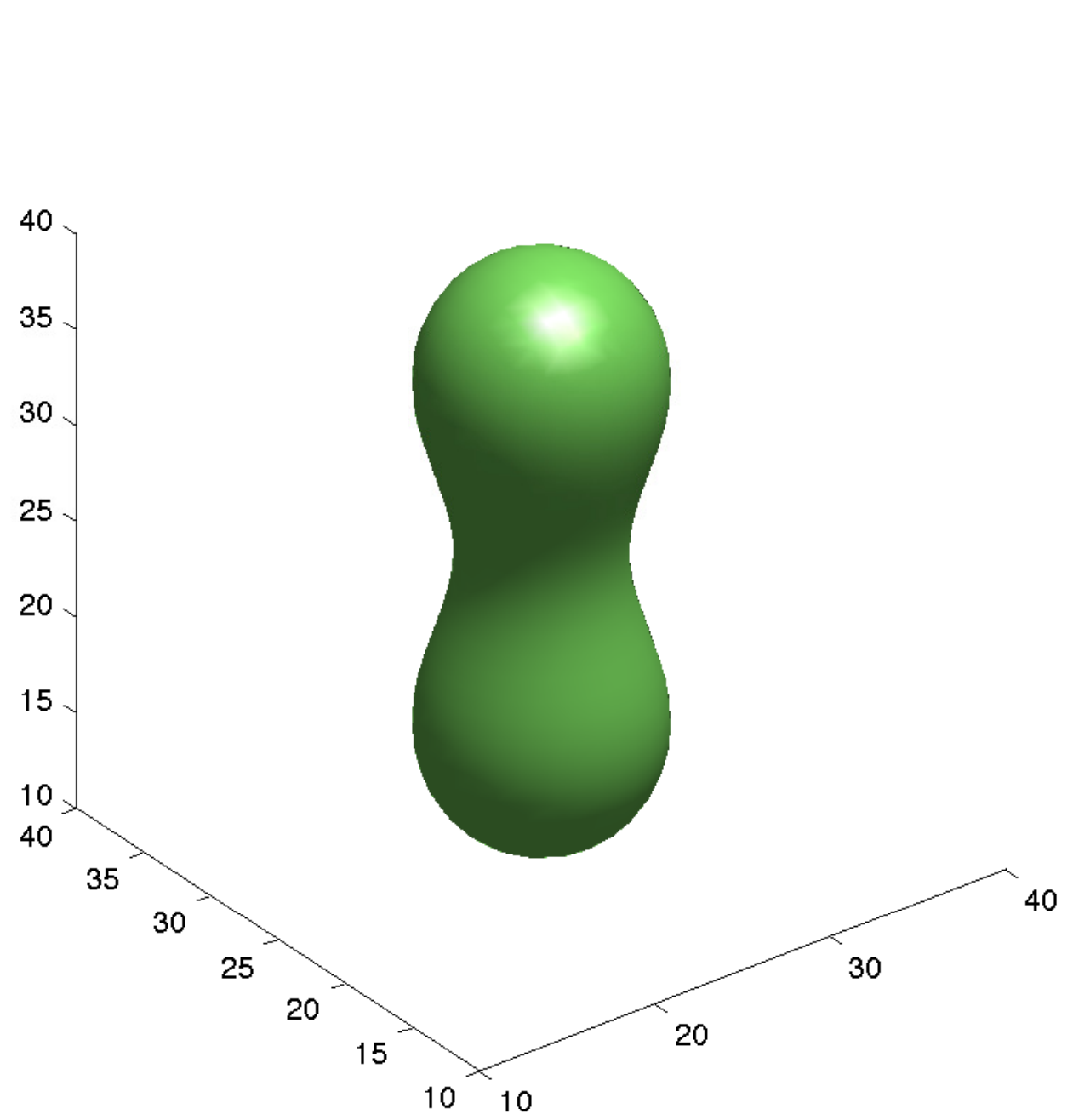}
\includegraphics[width=0.24\textwidth]{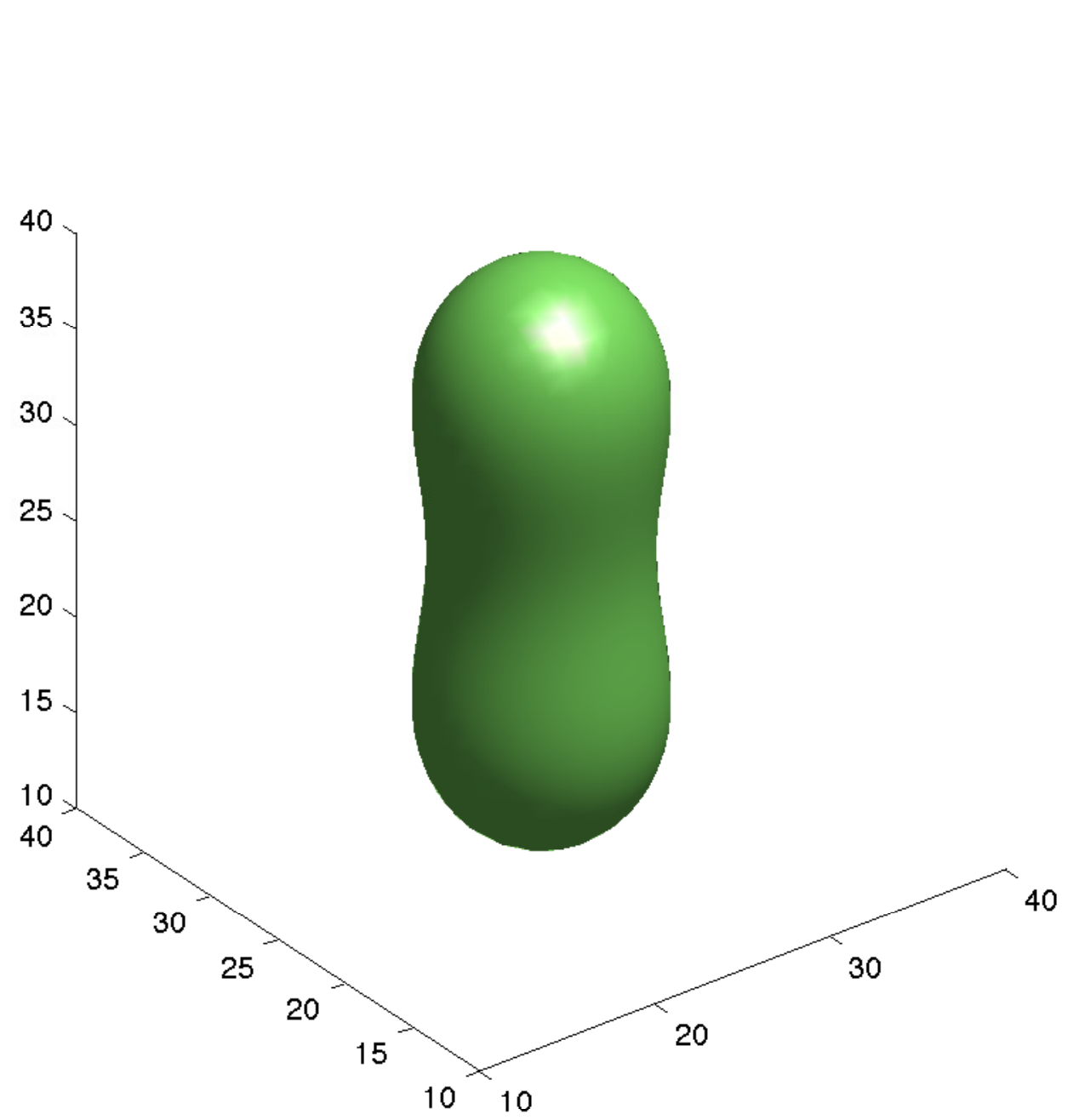}
\caption{\footnotesize Evolution of spheres under $L^2$ gradient flow for Bellettini's approximation of Willmore energy, with a bulk energy term encouraging expansion added.
The spheres expand and touch, and unlike in the 2D experiment with disks shown in Figure \ref{fig:bellettini_disks} or the 3D experiment with cylinders shown in Figure \ref{fig:cylinders}, the topological change takes place: the spheres merge, and the surface becomes instantaneously regular.}
\label{fig:bellettini_spheres}
\end{center}
\end{figure}


\section{Modified diffuse-interface Willmore flow}

In the following, we consider the modified diffuse approximation of
the Willmore energy \eqref{eq:def-mod-W} and the corresponding
$L^2$-gradient flow.

This modification by an additional `penalty' term offers some extra
flexibility compared to alternative energies such as
\eqref{eq:bell}. As long as solutions are close to the optimal-profile
construction the correction is negligible. One could therefore `switch
off' the additional term as long as $\mathcal{A}_\eps$ is small, or
only use the extra forcing in regions where $\mathcal{A}_\eps$ is
large. This might help to reduce computational costs, as the remaining
standard term is much easier to deal with. If solutions start to
deviate from the optimal profile, only the property that the extra
term blows up could be used. One therefore might need much less
accuracy in computing this term and might be able to choose different
numerical relaxation or discretization parameters than for the term
$\W_\eps$ in the total energy $\F_\eps$. Here however we only aim at a
proof of concept and do not exploit such possibilities or quantify
this hypothesis.

\subsection{Modified Willmore energy}

In order obtain a simpler variational derivative than we would obtain
from \eqref{eq:energyAdd}, we assume $|\eps \nabla u| \approx
\sqrt{2W(u)}$ and introduce an additional energy 
\begin{equation}
  \label{eq:energyAdd2}
  \tilde{\mathcal{A}}_\eps(u) :=
  \frac1{2\eps^{1+\alpha}}\int_\Omega\bigg(\underbrace{
      -\eps\Delta u + \eps^{-1}W'(u)}_{=:w} + \underbrace{
      \sqrt{2W(u)}\nabla \cdot \frac{\nabla u}{|\nabla
      u|}}_{=:v}\bigg)^2 \;\text{d}x,
\end{equation}
where $\alpha \le 1$. Whereas the analysis in Section \ref{sec:diff_new} does not cover this case, our numerical simulations show that this choice works equally well. In the following we therefore will use the diffuse Willmore energy
\begin{equation*}
  \tilde{\mathcal{F}}_\eps(u) := \mathcal{W}_\eps(u) + \tilde{\mathcal{A}}_\eps(u)
\end{equation*}
and will analyze numerically this functional and the corresponding $L^2$-gradient flow. We expect that a Gamma-convergence result as in Theorem \ref{thm:conv} also holds for $
\tilde{\mathcal{F}}_\eps$, but  the analysis in this case is much more difficult and will be subject to future investigations.

\subsection{Computation of the additional energy}

We compute the additional energy 
\begin{equation*}
  a(t) := \tilde{\mathcal{A}}_\eps(u_h(\cdot , t))
\end{equation*}
with $\alpha = 0$ during a simulation similar to the one
in Fig. \ref{fig1}. The only difference is that we have used 4th order
finite elements. Furthermore, we have regularized
\begin{equation*}
  \frac1{|\nabla u|} \leadsto \frac1{\sqrt{|\nabla u|^2 +
      \delta}}
\end{equation*}
with $\delta = 0.1$ in order not to divide by zero. Again we use a
uniform grid with $65^2=4225$ vertices leading to $66049$ degrees of
freedom for each unknown. In Fig. \ref{fig4} (left), one can see the
energy decrease of $e(t) = \mathcal{W}_\eps(u_h(\cdot, t))$ during
time with values very close to zero for large times. Fig. \ref{fig4}
(right) displays the expected blow up of the additional energy $a(t) =
  \tilde{\mathcal{A}}_\eps(u_h(\cdot, t))$ when the interfaces begin
  to ``feel'' each other. This behavior serves as a motivation to
  study the modified flow of the energy $\mathcal{W}_\eps +
  \tilde{\mathcal{A}}_\eps$ in the following.

\begin{figure}[here]
\begin{center}
\includegraphics*[width=0.49\textwidth]{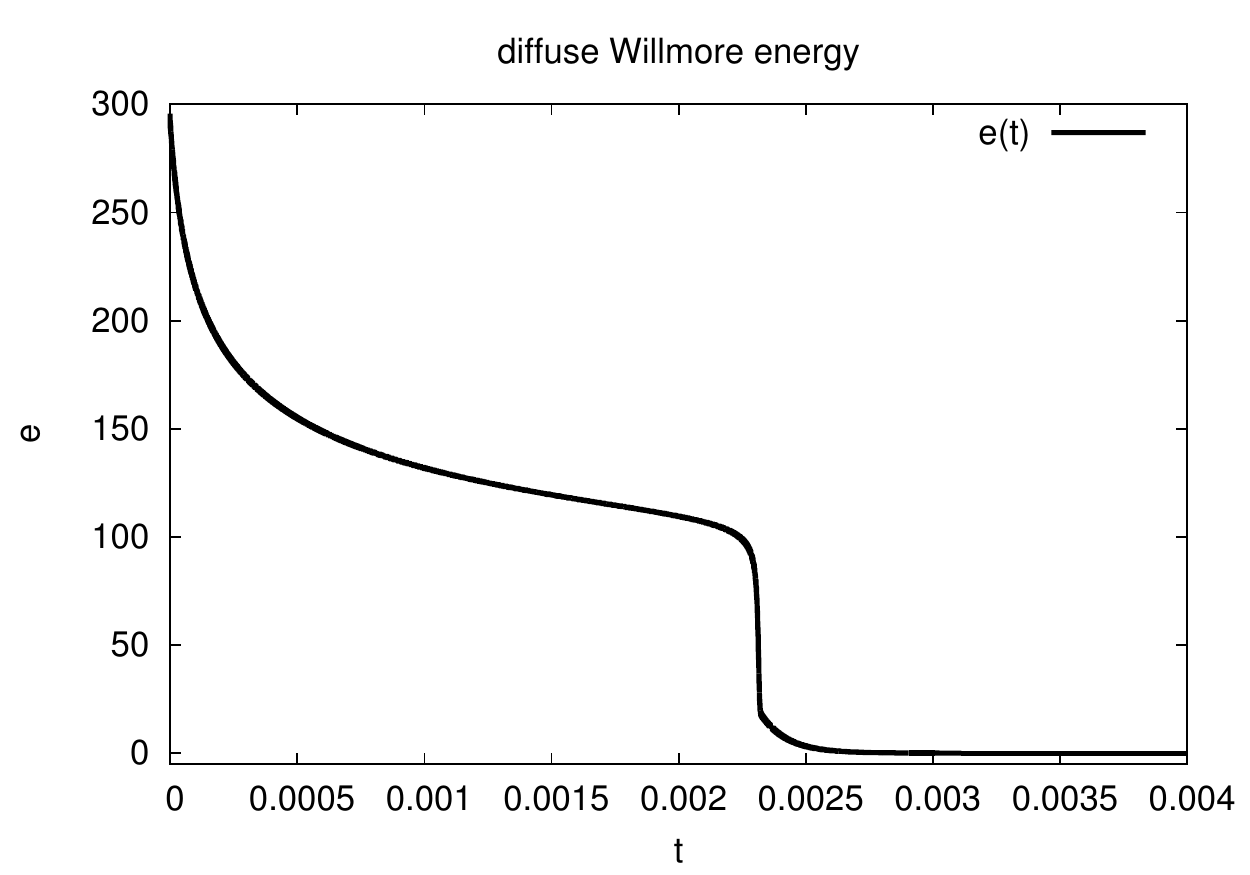}
\includegraphics*[width=0.49\textwidth]{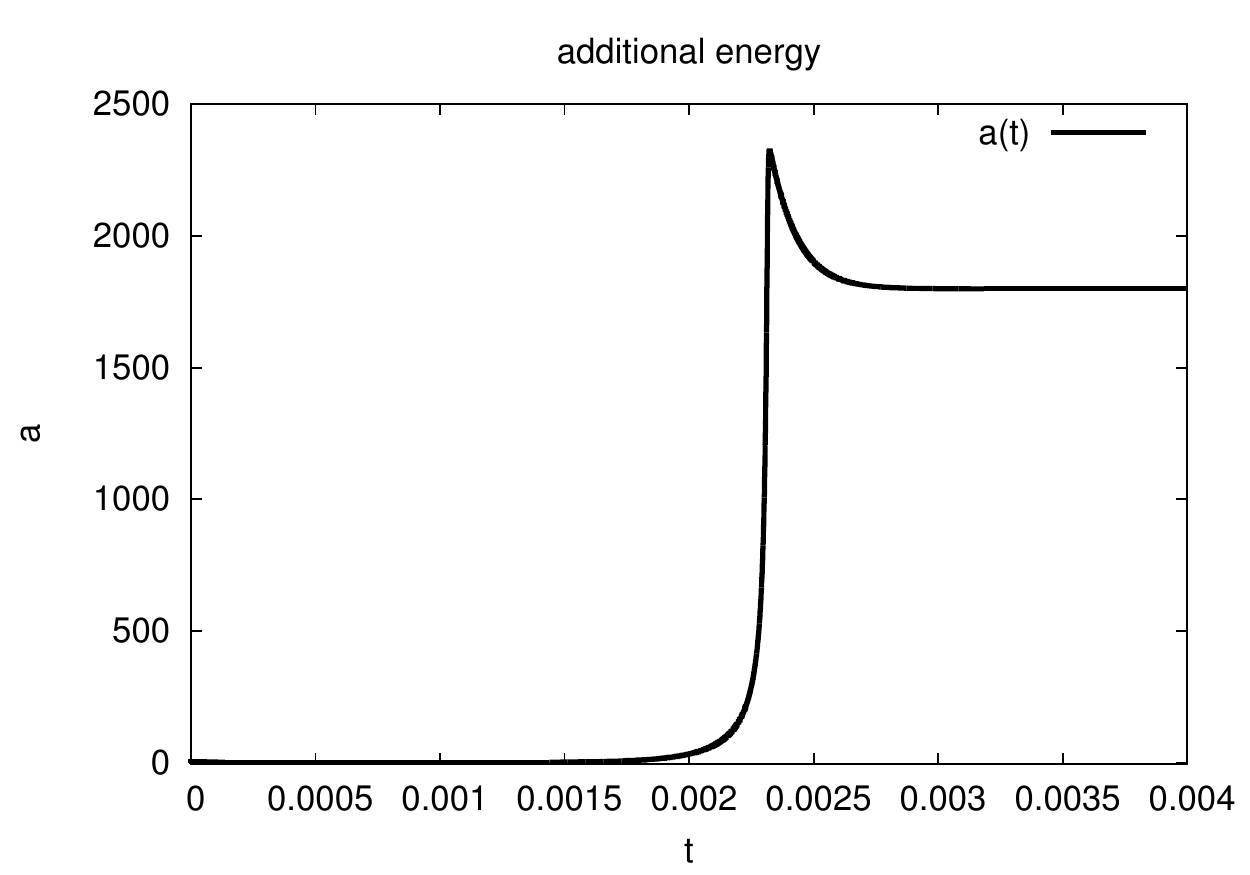}
\caption{\footnotesize \label{fig4} Evolution of ``standard'' diffuse-interface
  Willmore flow \eqref{eq:Wf-d}: Willmore energy $e(t) =
  \mathcal{W}_\eps(u_h(\cdot, t))$ (left), additional energy $a(t) =
  \tilde{\mathcal{A}}_\eps(u_h(\cdot, t))$ (right) versus time $t$.}
\end{center}
\end{figure}

\subsection{Modified flow}

Here, we consider the $L^2$-gradient flow 
\begin{equation}
  \label{eq:modflow}
  \eps\pd_t u = - \frac{\delta(\mathcal{W}_\eps +
    \tilde{\mathcal{A}}_\eps)}{\delta u}
\end{equation}
of the energy $\mathcal{W}_\eps(u) + \tilde{\mathcal{A}}_\eps(u)$. We
introduce
\begin{equation*}
  v :=  \sqrt{2W(u)}\nabla \cdot \frac{\nabla u}{|\nabla u|} =
  \nabla \cdot \left(\sqrt{2W(u)}\frac{\nabla u}{|\nabla u|}\right) -
  \big(\sqrt{2W(u)}\big)'|\nabla u|
\end{equation*}
and obtain the variational derivative
\begin{align}
  \label{eq:}
  \frac{\delta\tilde{\mathcal{A}}_\eps}{\delta u}
  &= \frac1{\eps^{1+\alpha}} \Big(-\eps \Delta (v + w) +
  \eps^{-1}W''(u)(v + w) + \frac{\big(\sqrt{2W(u)}\big)'}{\sqrt{2W(u)}}v(v + w) \\
  & \quad + \nabla \cdot (B(u, \nabla u)(\nabla v + \nabla w))
  \Big)\nonumber,
\end{align}
where
\begin{equation*}
  B(u, \nabla u) := \frac{\sqrt{2W(u)}}{|\nabla u|}\left(I - \frac{\nabla
    u}{|\nabla u|}\otimes\frac{\nabla u}{|\nabla u|}\right).
\end{equation*}
Thereby, $I = (\delta_{ij})_{ij}$ denotes the unit matrix.

 In order to numerically treat the modified flow \eqref{eq:modflow}, we
split the time interval $[0,T]$ by discrete time instants $0 = t_0 < t_1 <
\dots < t_{M} = T$, from which one gets the time steps $\Delta t_m := t_{m
  + 1} - t_m$, $m = 0, 1, \dots, M-1$. Moreover, we apply the
following operator splitting type Ansatz for time-discrete functions
$u^{(m)}$, $v^{(m)}$, $w^{(m)}$ at time instant $t_m$: 
\begin{enumerate}
\item compute $v^{(m)}$ via
  \begin{equation}
    \label{eq:modflow_v}
    v^{(m)} = \nabla \cdot \left(\frac{\sqrt{2W(u^{(m)})}}{|\nabla
        u^{(m)}|}\nabla u^{(m)}\right) - \left(\sqrt{2W(u^{(m)})}\right)'|\nabla u^{(m)}|
  \end{equation}
\item solve for $u^{(m+1)}$ and $w^{(m+1)}$ in 
  \begin{align}
    \label{eq:modflow_w}
    \eps\frac{u^{(m+1)} - u^{(m)}}{\Delta t_m} 
    &= \Delta w^{(m+1)} - \eps^{-2} W''(u^{(m)}) w^{(m+1)}\\
    &\quad - \frac1{\eps^{1+\alpha}} \Big(-\eps \Delta (v^{(m)} + w^{(m+1)})\nonumber\\
    &\quad  +
    \eps^{-1}W''(u^{(m)})(v^{(m)} + w^{(m+1)})\nonumber\\
    &\quad+ \frac{\sqrt{2W(u^{(m)})}'}{\sqrt{2W(u^{(m)})}}v^{(m)}(v^{(m)} +
    w^{(m+1)}) \nonumber\\ 
    & \quad + \nabla \cdot (B(u^{(m)}, \nabla u^{(m)})(\nabla v^{(m)}
    + \nabla w^{(m+1)})) \Big)\nonumber\\    
    \label{eq:modflow_u}
    w^{(m+1)} &= -\eps \Delta u^{(m+1)} +
    \eps^{-1}W''(u^{(m)})u^{(m+1)}\\
    & \quad + \eps^{-1}W'(u^{(m)}) - \eps^{-1}W''(u^{(m)})u^{(m)}.\nonumber
  \end{align}
\end{enumerate}
Thereby, for $\delta, \delta_W > 0$, we regularize
\begin{align*}
  \frac1{|\nabla u^{(m)}|} &\leadsto \frac1{\sqrt{|\nabla u^{(m)}|^2 +
      \delta}},\\
  \frac{1}{\sqrt{2W(u^{(m)})}} &\leadsto
  \frac{1}{\sqrt{2W(u^{(m)})}+\delta_W}
\end{align*}
for all terms of these forms appearing in
\eqref{eq:modflow_v}--\eqref{eq:modflow_u}. In addition to the
previous parameters, we have used $\delta_W = 0.01$. 

As a first test, we observe that the modified flow yields a reasonable
approximation of Willmore flow in the case of a growing circle. In
Fig. \ref{fig:circle}, we compare numerical results with the analytic
expression for a circle growing according to Willmore flow. Thereby,
we use an initial radius $R(0)=0.1$ and plot the radius $R(t)$ versus
time.

\begin{figure}[here]
\begin{center}
\includegraphics*[width=0.49\textwidth]{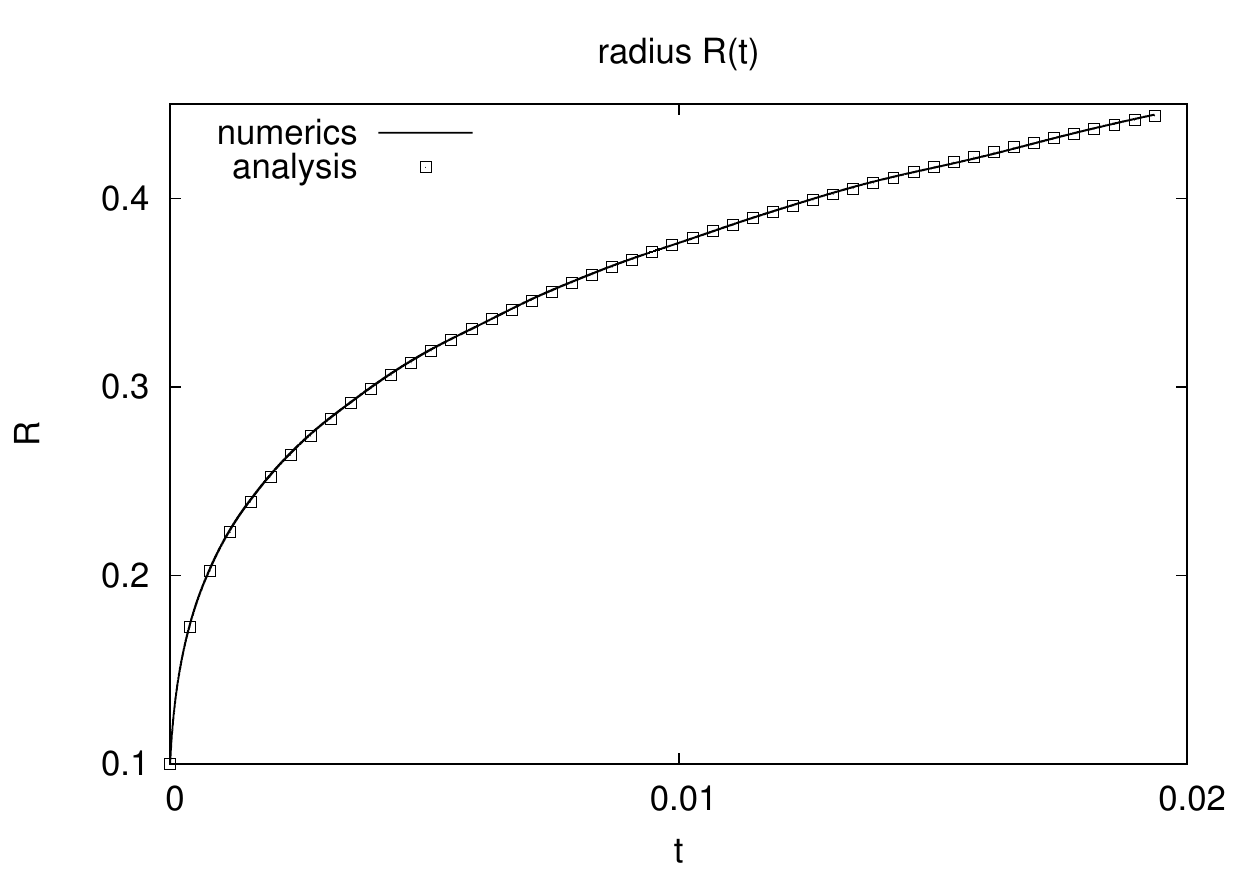}
\caption{\footnotesize \label{fig:circle} Evolution of modified flow
  \eqref{eq:modflow}: Radius of growing circle versus time: Analytic
  expression and results of modified diffuse-interface
  flow. \eqref{eq:modflow}}
\end{center}
\end{figure}

In Fig. \ref{fig:mod9}, we see numerical results for
this flow showing the phase-field variable $u$ for a symmetric initial
condition with nine circles with equal radii $0.1$ as in
Fig. \ref{fig1}. The final picture shows the nearly stationary
solution. Fig. \ref{fig:mod9energy} shows the energy decrease for this
example. 

\begin{figure}[here]
\begin{center}
\includegraphics*[width=0.24\textwidth]{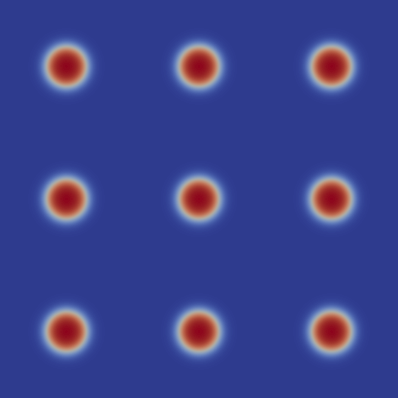}
\includegraphics*[width=0.24\textwidth]{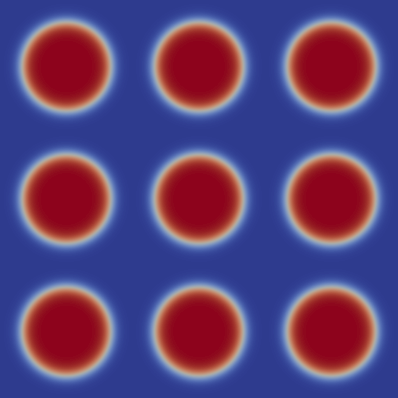}
\includegraphics*[width=0.24\textwidth]{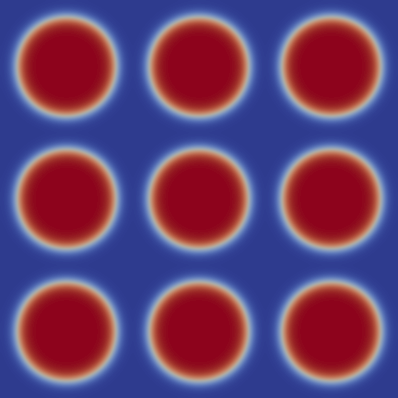}
\includegraphics*[width=0.24\textwidth]{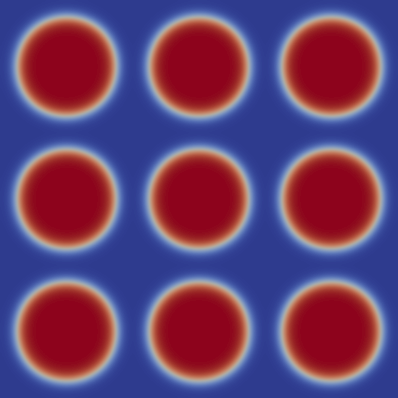}
\caption{\footnotesize \label{fig:mod9} Evolution of modified flow
  \eqref{eq:modflow}: Discrete phase-field $u_h$ for different times
  $t=0$, $t\approx0.0012$, \mbox{$t\approx0.0024$,} $t\approx0.0045$.}
\end{center}
\end{figure}

\begin{figure}[here]
\begin{center}
\includegraphics*[width=0.5\textwidth]{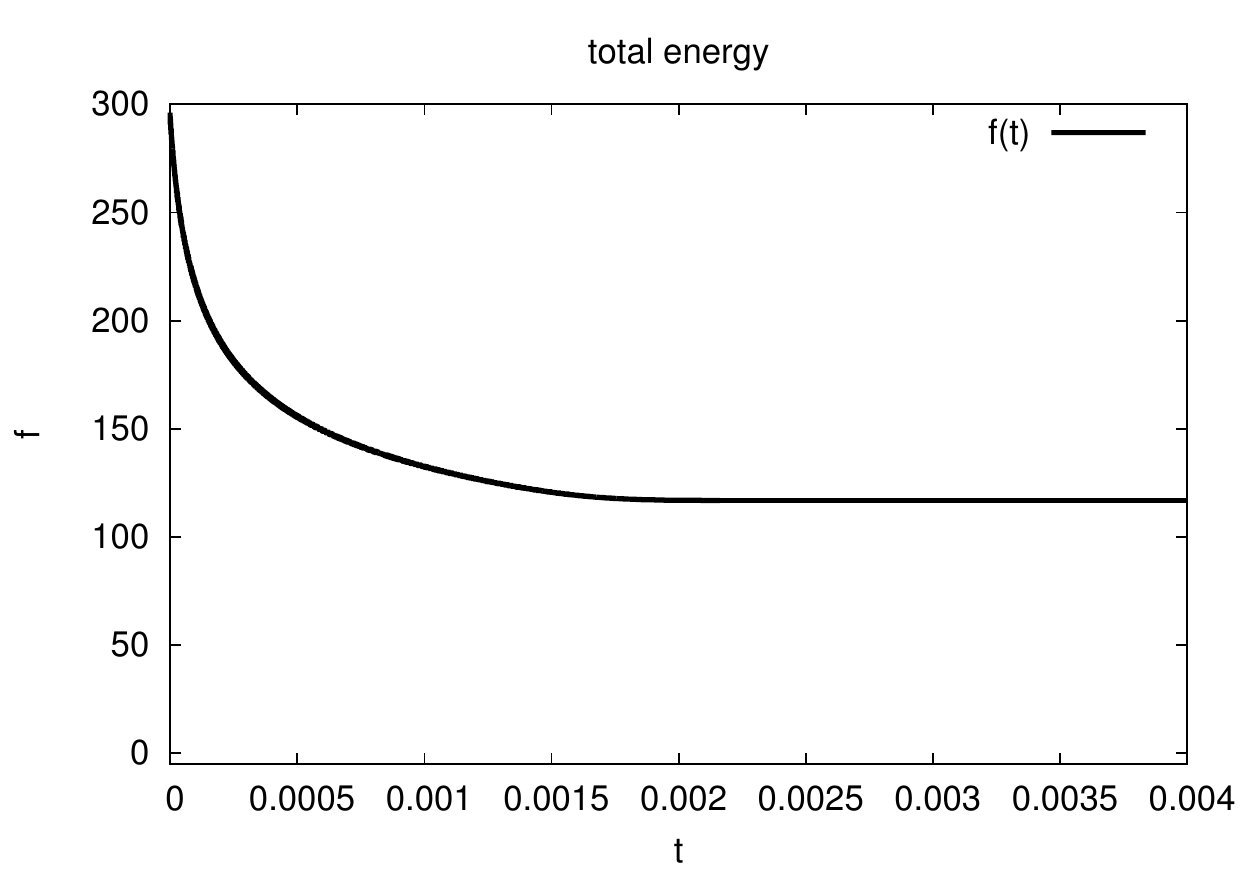}
\caption{\footnotesize \label{fig:mod9energy} Evolution of modified flow
  \eqref{eq:modflow}: Total energy $f(t) =
  \mathcal{W}_\eps(u_h(\cdot, t)) +
  \tilde{\mathcal{A}}_\eps(u_h(\cdot, t))$ versus time $t$.}
\end{center}
\end{figure}

In Fig. \ref{fig:modflowNonsymmPer} the nonsymmetric initial condition
from Fig. \ref{fig2} has been used. Again, the nearly
stationary discrete solution at time $t \approx 0.016$ shows that self
intersections are prohibited for this modified flow.

\begin{figure}[here]
\begin{center}
\includegraphics*[width=0.24\textwidth]{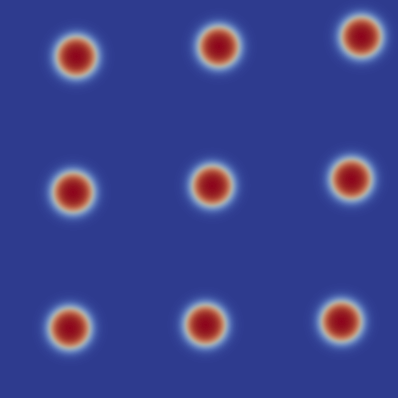}
\includegraphics*[width=0.24\textwidth]{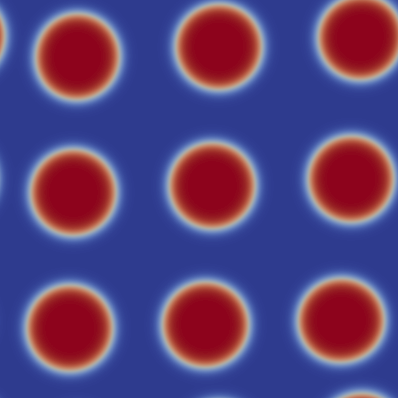}
\includegraphics*[width=0.24\textwidth]{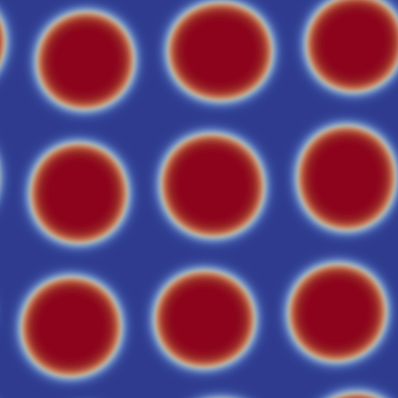}
\includegraphics*[width=0.24\textwidth]{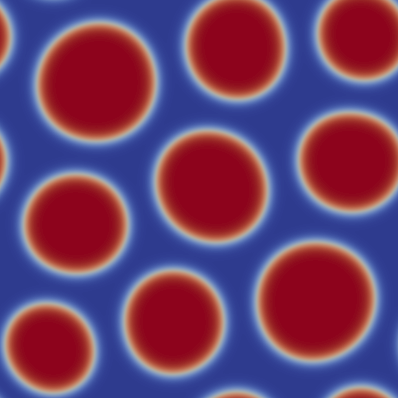}
\caption{\footnotesize \label{fig:modflowNonsymmPer} Evolution of modified flow
  \eqref{eq:modflow}: Discrete phase-field $u_h$ for different times
  $t=0$, $t\approx0.0010$, \mbox{$t\approx0.0030$,} $t\approx0.1046$.}
\end{center}
\end{figure}

In Fig. \ref{fig:mod2} the initial condition with two circles
is not symmetric as in Fig. \ref{fig:wil2}. The circles grow until the
interfaces ``feel'' each other. The modified energy prevents the level
sets from forming self intersections. The contour plots in
Fig. \ref{fig:mod2} are taken at similar times as in
Fig. \ref{fig:wil2}.

\begin{figure}
\begin{center}
\includegraphics*[width=0.24\textwidth]{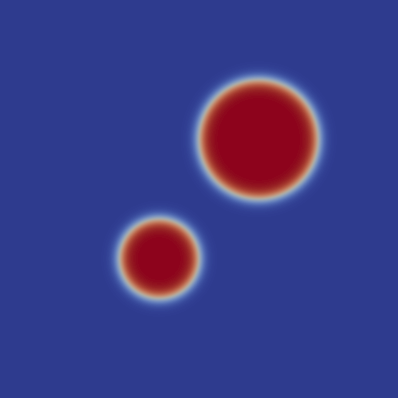}
\includegraphics*[width=0.24\textwidth]{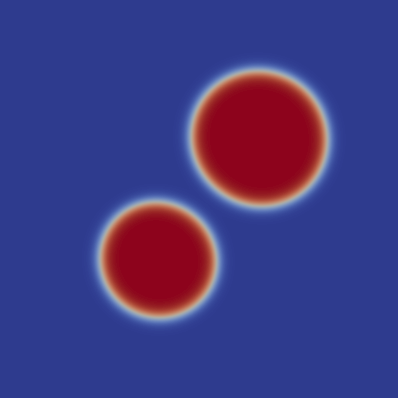}
\includegraphics*[width=0.24\textwidth]{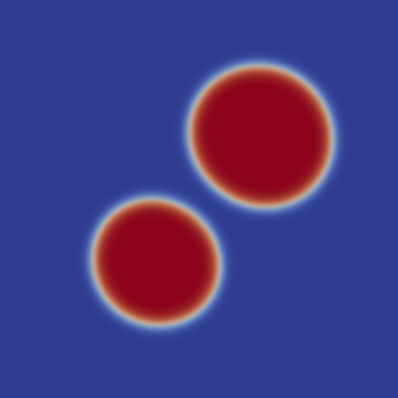}
\includegraphics*[width=0.24\textwidth]{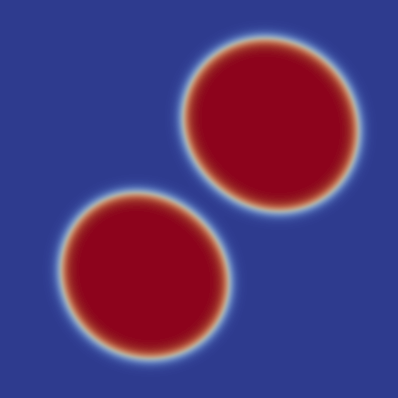}
\caption{\footnotesize \label{fig:mod2} Evolution of modified flow
  \eqref{eq:modflow}: Discrete phase-field $u_h$ for different times
  $t=0$, $t\approx0.003$, \mbox{$t\approx0.0049$,} $t\approx0.0199$.}
\end{center}
\end{figure}

In the gradient flow simulations for Bellettinis energy and the modified energy $\mathcal{F}_\eps$ we have often observed that circles upon collision keep touching and evolve to an ellipse type shape. We expect that such shapes represent suitable elastica. In a final example, we therefore compare (nearly) stationary states in our numerical simulations with graphs of minimal elastic energy \cite{LiJe07} that present possible optimal configurations. In \cite{LiJe07},
Linn{\'e}r and Jerome prove the existence of a unique graph of
minimal elastic energy among $W^{2,2}$-regular graphs over $[0,1]$ that have finite length, that start with horizontal slope in $(x,y)=(0,0)$ and reach the line $\{x=1\}$ with vertical slope. Moreover, Linn{\'e}r and Jerome provide an explicit formula for this minimizer that can be compared with shapes that we obtain as stationary configurations in our simulations. With this aim we have chosen the rectangular domain $\Omega =
(-1.1,1.1) \times (-2.2,2.2)$ with periodic boundary conditions and an ellipse as initial condition, such
that we expect the level set $\Gamma_h := \{u_h = \frac{1}{2}\}$ to
stop growing in the $x$-direction at $x \approx \pm 1$. In Fig. \ref{fig:LiJe07} we see
level sets $\{u_h = \frac{1}{2}\}$ at times $t=0$ and $t \approx
2.2462$, where the configurations have become nearly stationary. By symmetry, at the touching point with the vertical boundary of $\Omega$ the level line $\Gamma_h$ has vertical slope and its lower right quarter therefore can be compared to appropriately shifted versions of the analytic minimizer from \cite{LiJe07}. First we consider a shift such that the line $\{x=1\}$ is reached at $(x,y)=(1,0)$, see the curve $\Gamma_1$ in Fig. \ref{fig:LiJe07} (left). Second we consider a shift such that the analytic minimizer and $\Gamma_h$ (at the final time) agree at the bottom point, see the curve $\Gamma_2$ in Fig. \ref{fig:LiJe07} (right). We observe a pretty good agreement between the final configuration of our simulations and the analytic minimizer, indicating that equilibrium configurations of our simulations in fact consist of unions of elastica.
\begin{figure}
\begin{center}
\includegraphics*[height=0.3\textheight]{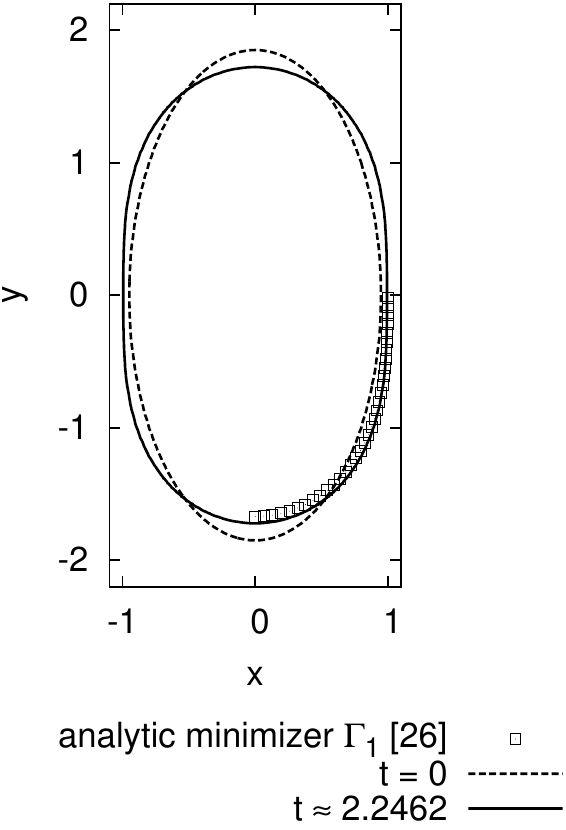}
\includegraphics*[height=0.3\textheight]{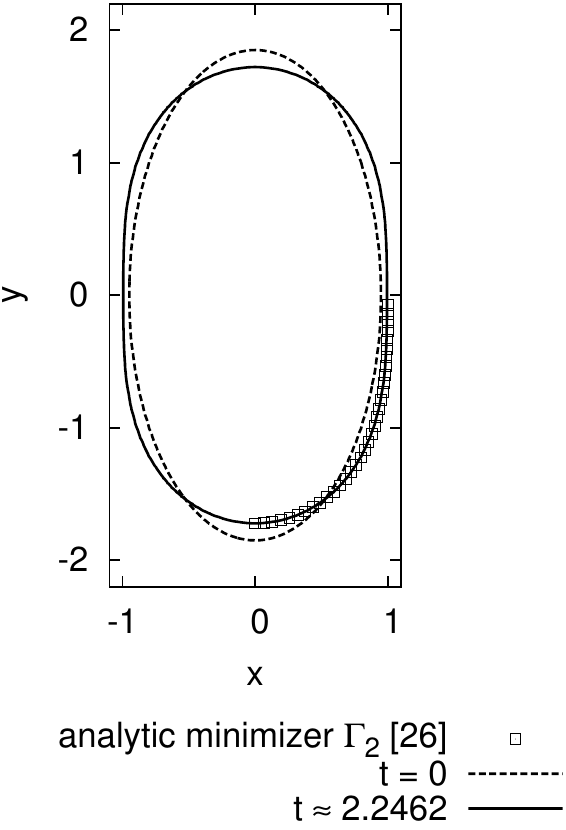}
\caption{\footnotesize \label{fig:LiJe07} Evolution of modified flow
  \eqref{eq:modflow}. Left: Level curve $\{u_h=1/2\}$ at times $t=0$
  and $t \approx 2.2462$ and plot of the shifted version $\Gamma_1$ of the analytic minimizer from
  \cite{LiJe07}. Right: level curve $\{u_h=1/2\}$ at times $t=0$ and $t \approx
  2.2462$ and shifted version $\Gamma_2$ of the analytic minimizer from \cite{LiJe07}.}
\end{center}
\end{figure}

%
\section{Discussion}
\label{sec:discussion}
We have analyzed the behavior of different diffuse approximations of
the Willmore flow in situations where diffuse interfaces
collide. Different scenarios emerge. In level set approximations the
respective phases typically merge (not investigated here, but see the results presented in \cite{DrRu04}) at the moment of collision. For the standard diffuse approximation on the
other hand we have demonstrated that (at least in two dimensions) transversal intersections of diffuse interfaces are preferred. As an alternative we have introduced two new diffuse approximations with again different behavior: after collision the phases keep touching, away from the touching points the interfaces evolve to elastica. 

Whereas all approximations converge to the same evolution as long as
phases remain well separated, the kind of approximation determines the
evolution past collision of interfaces. 
Here is a summary of these behaviors:

\begin{center}
\begin{tabular}{|c|c|c|}
\hline
Method & Disks in $\mathbb{R}^2$ & Spheres in $\mathbb{R}^3$\\
\hline
Standard & “Cross” formation & Merger\\
Level set & Merger & Merger\\
Bellettini & No merger & Merger\\
New approx. & No merger & Not investigated (we expect mergers)\\
\hline
\end{tabular}
\end{center}\medskip

\noindent Although there appears to be agreement in the case of spheres, it is reasonable to assume that the discrepancies in the two dimensional handling of topological changes would manifest themselves also in 3D with more general surfaces.

The expected behavior through these topological changes, and therefore the choice of the approximation, may depend on the specific application.
For example, in image processing applications such as inpainting (see e.g. \cite{Bertalmio2000}), it is often desired that colliding interfaces merge in 2D.
Indeed, in the important work \cite{ChanShenKang02}, the Willmore energy is utilized to drive interfaces towards collisions and mergers. 
In other applications, for instance in the segmentation of medical images (see e.g. \cite{AlexandrovSantosa2005}), it may be desired to prevent merging of colliding interfaces.

Our new approximation and that of Bellettini \cite{Bell97} agree in the handling of topological changes considered here for Willmore flow.
Indeed, for these Gamma-convergent approximations, we argue that the observed limit flow
describes the only reasonable evolution that is continuous with
respect to the $L^1$-topology and keeps decreasing the Willmore
energy. In contrast, typical `mergers' do not share this property. In
fact, let us assume that we have a sequence of Jordan curves
$(\gamma_k)_{k\in\N}$ that approximates the union of two touching
balls $$S\,=\, \{(x,y)\in\R^2 \,:\, (x+1)^2+y^2\,<\, 1\}\cup
\{(x,y)\in\R^2 \,:\, (x-1)^2+y^2\,<\, 1\}$$ in $L^1$-distance. Assume for
simplicity that the $\gamma_k$ are symmetric with respect to the
$x$-axis, that the upper half is given as a nonnegative graph, and that the convergence is in an $C^0$-sense outside a region $\{(x,y)\in\R^2\,:\, |x|<1-\delta\}$. We then can
modify the curves by replacing the part below the $x$-axis by a
circular arc that touches $\gamma_k$ in its intersection points with
the $x$-axis (see Fig. \ref{fig:gammak}).
\begin{figure}
\begin{center}
\includegraphics*[width=0.45\textwidth]{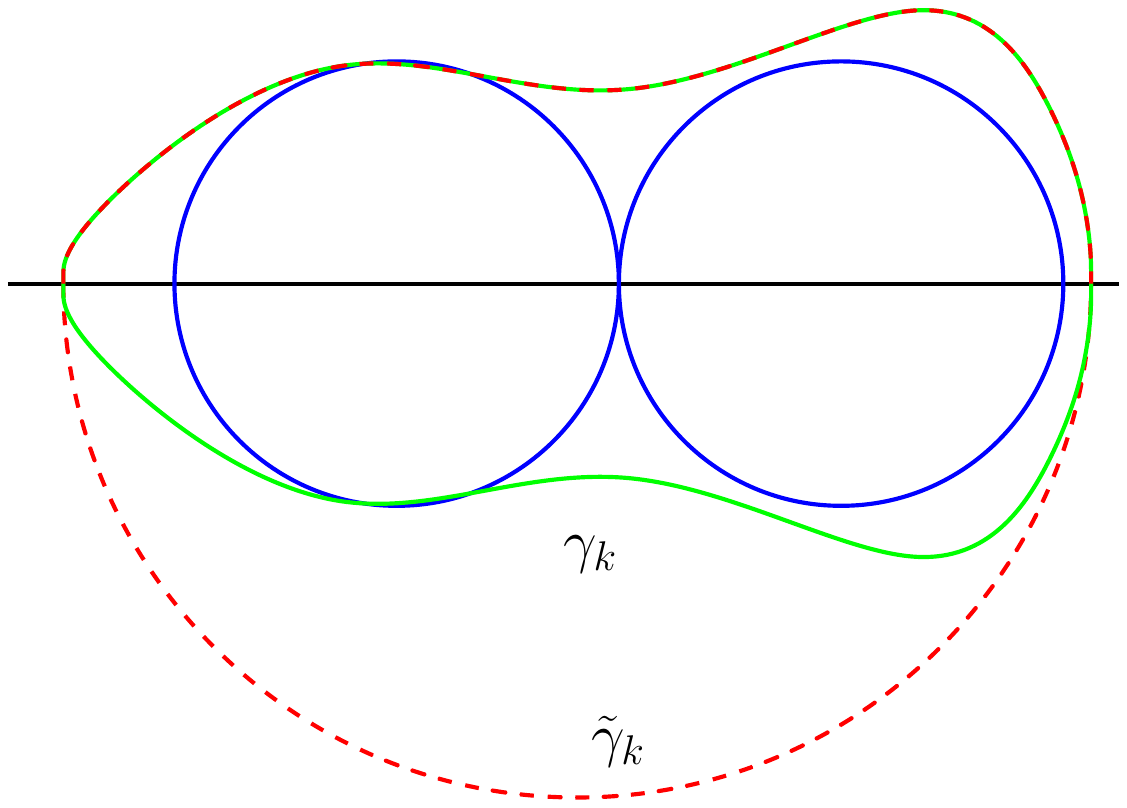}
\includegraphics*[width=0.45\textwidth]{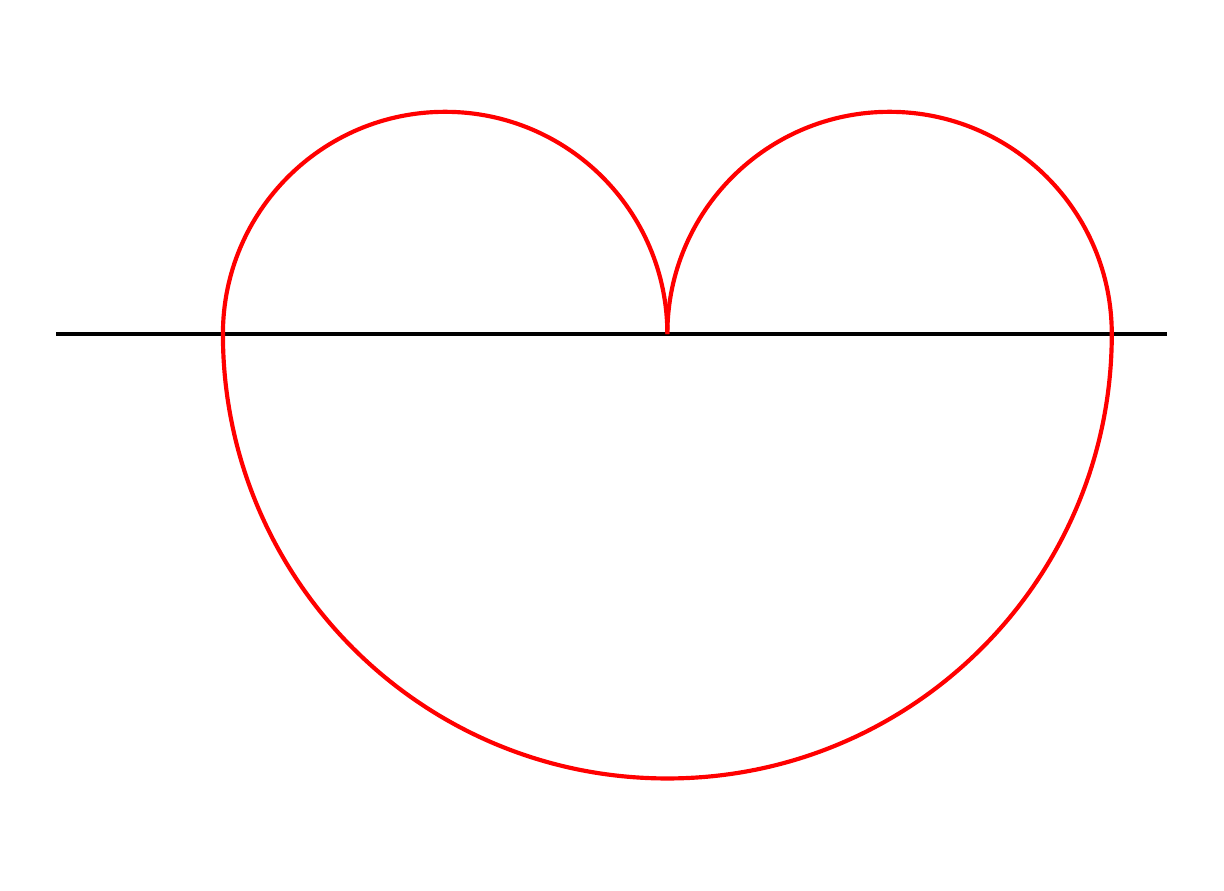}
\caption{\footnotesize \label{fig:gammak} Left: Curve $\gamma_k$ and
  modification $\tilde{\gamma}_k$. Right: Limit of the modified curves.}
\end{center}
\end{figure}
 By construction the modified curves $\tilde{\gamma}_k$ are $C^1$-Jordan curves that are piecewise $C^2$. As the radius of the attached half-circles approaches $2$ with $k\to\infty$ there exists a constant $C>0$ independent of $k$ such that
\begin{gather}
	\W(\tilde{\gamma}_k) \,\leq\, \W(\gamma_k) +C. \label{eq:no-mergers}
\end{gather}
On the other hand the sets enclosed by $\tilde{\gamma}_k$ converge in $L^1$-distance to $\big(S\cap \overline{\R^2_+}\big)\, \cup\,\big( B(0,2)\cap \R^2_-\big)$, where $\R^2_\pm$ denote the upper and lower half-plane, respectively. This however is a set with exactly one simple cusp point. By \cite[Theorem 6.4]{BeDP93} this implies that
\begin{gather*}
	\lim_{k\to\infty} \W(\tilde{\gamma}_k)\,=\, \infty.
\end{gather*}
From \eqref{eq:no-mergers} we deduce therefore that the elastica
energy of $\gamma_k$ blows up, too.
%
\bigskip

In 3D the behavior of diffuse Willmore flows in situations where diffuse interfaces collide is different from the two dimensional case: two spheres that are initially disjoint but forced to come in contact via a volumetric expansion term are likely to merge under gradient descent for the $L^1$ relaxation of Willmore energy.
Indeed, using e.g. catenoid “necks”, of the form
\begin{equation*}
y = a \cosh \left( \frac{x-b}{a} \right),
\end{equation*}
with $a$ and $b$ are appropriately chosen, we can “connect” two spheres at the moment of contact with an arbitrarily small neck that is tangent to the spheres after slightly shifting them if necessary, while decreasing the energy.
Note that the catenoid neck, regardless of its scale, has no elastica energy at all, since it happens to be a minimal surface and thus has vanishing mean curvature.
To be more precise, the foregoing discussion shows that it is easy to construct a continuous in $L^1$ map from the interval $[0,1]$ into sets in $\mathbb{R}^3$ that deforms the union of two spheres in contact to a $C^2$ surface topologically equivalent to the sphere, while decreasing the Willmore energy during the deformation. 

\medskip

{\bf Acknowledgement.} The authors would like to thank Luca Mugnai for
stimulating discussions.

\medskip


\end{document}